\documentclass[11pt,twoside]{gtpart}

\usepackage{amsfonts, amsmath,amssymb, amscd}
\usepackage{hyperref}
\usepackage[all]{xy}
\usepackage[lite]{amsrefs}
\usepackage{amsthm}
\usepackage{geometry}
\usepackage{graphicx}
\usepackage{wrapfig}
\usepackage{subfig}
\usepackage{pstricks, pstricks-add}

\usepackage{graphicx}

\newcommand\la{\lambda}

\newcommand{\CC}{\ensuremath{\mathbb{C}}}
\newcommand{\RR}{\ensuremath{\mathbb{R}}}
\newcommand{\ZZ}{\ensuremath{\mathbb{Z}}}
\newcommand{\QQ}{\ensuremath{\mathbb{Q}}}

\newcommand{\NN}{\ensuremath{\mathbb{N}}}

\newcommand{\PP}{\ensuremath{\mathbb{P}}}

\newcommand{\HHH}{\ensuremath{\mathcal{H}}}

\newcommand{\ra}{\ensuremath{\rightarrow}}

\newcommand\Oh{{\mathcal O}}

\newcommand\sX{{\mathcal X}}

\newcommand{\bN}{{\mathbb N}}
\newcommand{\bZ}{{\mathbb Z}}
\newcommand{\bC}{{\mathbb C}}

\newcommand{\bQ}{{\mathbb Q}}

\newcommand{\cutoff}[1]{}

\newcommand{\Hy}{\mathcal{H}}
\newcommand{\ringO}{\mathcal{O}}

\newcommand{\one}{{\mathrm{Id}}}

\newcommand{\rationals}{{\mathbb{Q}}}

\newcommand*{\Homol}{\operatorname{H}}
\newcommand{\age}{{\rm shift}}

\newcommand{\codim}{{\rm codim}_\C}
\newcommand{\suchthat}{ \medspace | \medspace}
\newcommand{\PSO}{\mathrm{PSO}}
\newcommand{\PSU}{\mathrm{PSU}}
\newcommand{\PSL}{\mathrm{PSL}}

\renewcommand{\geq}{\geqslant}

\newcommand{\Afour}{\mathcal{A}_4}
\newcommand{\Sthree}{\mathcal{D}_3}
\newcommand{\Kleinfourgroup}{\mathcal{D}_2}
\newcommand{\Dtwon}{\mathcal{D}_{n}}

\theoremstyle{plain}
\newtheorem{thm}{\bfseries Theorem}
\newtheorem{theorem}[thm]{\bfseries Theorem}
\newtheorem{Lem}[thm]{\bfseries Lemma}
\newtheorem{lemma}[thm]{\bfseries Lemma}

\newtheorem{proposition}[thm]{\bfseries Proposition}

\newtheorem{corollary}[thm]{\bfseries Corollary}

\newtheorem{df}[thm]{\bfseries Definition}

\theoremstyle{remark}

\newtheorem{note}[thm]{\bfseries Note}

\newtheorem{remark}[thm]{\bfseries Remark}

\newtheorem*{ConditionA}{\bfseries Condition A}
\newtheorem*{ConditionB}{\bfseries Condition B}

\newtheorem{construction}[thm]{\bfseries Construction}
\newcommand{\cellCondition}{A}

 \newcommand{\fouredges}{\includegraphics{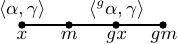}}

\newcommand{\circlegraph}{\includegraphics{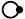}}

\newcommand{\edgegraph}{ \includegraphics{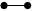}}

\newcommand{\graphFive}{  \includegraphics{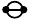}}

\newcommand{\graphTwo}{  \includegraphics{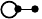}}

\title[On Ruan's Cohomological Crepant Resolution Conjecture 
for the complexified Bianchi orbifolds]{On Ruan's Cohomological
Crepant Resolution Conjecture \\ for the complexified Bianchi orbifolds}
\author{Fabio Perroni and Alexander D. Rahm}
\email{fperroni@units.it}
\urladdr{http://perroni.dmg.units.it/}
\address{University of Trieste, Department of Mathematics and Geosciences, Via A. Valerio 12/1, 34127 Trieste, Italia}
\email{Alexander.Rahm@uni.lu}
\urladdr{http://math.uni.lu/\char126rahm/}
\address{Universit\'e du Luxembourg, Mathematics Research Unit, 6, av. de la Fonte, L-4364 Esch-sur-Alzette, Luxembourg}
 
\date{\today}
\arxivreference{arXiv:1109.5923 [math.KT]}
\subject{primary}{msc2000}{55N32, Orbifold cohomology}
\begin{document}
\begin{abstract}
We give formulae for the Chen--Ruan orbifold cohomology for the orbifolds given by a Bianchi group acting on complex hyperbolic 3-space. \\
The Bianchi groups are the arithmetic groups $\text{PSL}_2(\mathcal{O})$, where $\mathcal{O}$ is the ring of integers in an imaginary quadratic number field. 
The underlying real orbifolds which help us in our study, given by the action of a Bianchi group on real hyperbolic 3-space (which is a model for its classifying space for proper actions),
have applications in physics.

We then prove that, for any such orbifold, its Chen--Ruan orbifold cohomology ring 
is isomorphic to the usual cohomology ring of any crepant resolution of its coarse moduli space.
By vanishing of the quantum corrections, we show that this result fits in with Ruan's Cohomological Crepant Resolution Conjecture. 
\end{abstract}
\maketitle

\setcounter{secnumdepth}{3}
\setcounter{tocdepth}{2}

\section{Introduction}

Recently motivated by string theory in theoretical physics,
a stringy topology of orbifolds has been introduced in mathematics \cite{AdemLeidaRuan}.
Its essential innovations consist of Chen--Ruan orbifold cohomology \cite{ChenRuan}, \cite{Orbifold_Gromov-Witten_theory}
 and orbifold $K$--theory.
 They are of interest as topological quantum field theories~\cite{Uribe_et_al}.
Ruan's cohomological crepant resolution conjecture~\cite{RuanCCRC}
    associates Chen--Ruan orbifold cohomology with the ordinary cohomology of a resolution of the singularities of the coarse moduli space of the given orbifold.
We place ourselves where the conjecture is still open: in three complex dimensions and outside the global quotient case.
There, we are going to calculate the Chen--Ruan cohomology of an infinite family of orbifolds in this article;
and prove in Section~\ref{ringIsomorphism} that it is isomorphic as a ring to the cohomology of their crepant resolution of singularities.

Denote by $\rationals(\sqrt{-m})$, with $m$ a square-free positive integer, an imaginary quadratic number field, and by $\ringO_{-m}$ its ring of integers.
The \emph{\mbox{Bianchi} groups} are the projective special linear groups 
$\mathrm{PSL_2}(\ringO_{-m})$, throughout the paper 
we denote them by $\Gamma$.
The \mbox{Bianchi} groups may be considered as a key to the study of a larger class of groups, the \emph{Kleinian} groups, which date back to work of Henri Poincar\'e~\cite{Poincare}.
In fact, each non-co-compact arithmetic Kleinian group is commensurable with some \mbox{Bianchi} group~\cite{MaclachlanReid}.
A wealth of information on the \mbox{Bianchi} groups can be found in the monographs \cites{Fine, ElstrodtGrunewaldMennicke, MaclachlanReid}.
These groups act in a natural way on real hyperbolic three-space
$\Hy^3_\R$, which is isomorphic to the symmetric space associated to them.
This yields orbifolds $[\Hy^3_\R / \Gamma]$ that are studied in 
cosmology~\cite{AurichSteinerThen}.

The orbifold structure obtained in this way
is determined by a fundamental domain and its stabilizers and identifications.
The computation of this information has been implemented for all Bianchi groups~\cite{BianchiGP}. 

In order to obtain complex orbifolds, we consider the 
complex hyperbolic three-space $\Hy^3_\C$.
Then $\Hy^3_\R$  is naturally embedded into $\HHH_\CC^3$ as the fixed points set 
of the complex conjugation (Construction~\ref{complexification}). 
The action of a Bianchi group $\Gamma$ on $\Hy^3_\R$
extends to an action on $\Hy_\C^3$, in a natural way.
Since this action is properly discontinuous (Lemma~\ref{ActionProperlyDiscontinuous}),
we obtain a complex orbifold $[\Hy^3_\C/\Gamma]$,
which we call a \textit{complexified Bianchi orbifold}.

\subsection*{The vector space structure of Chen--Ruan orbifold cohomology} 

Let~$\Gamma$ be a discrete group acting \thinspace \emph{properly discontinuously}, hence  with finite stabilizers, by bi-holo\-morphisms on a
complex  manifold~$Y$. For any element $g \in \Gamma$, denote by $C_\Gamma(g)$ the centralizer of $g$ in~$\Gamma$. Denote by $Y^g$ the subspace of 
$Y$ consisting of the fixed points of $g$.

\begin{df}[\cite{ChenRuan}]\label{HCR}
Let $T \subset \Gamma$ be a set of representatives of the conjugacy classes of elements of finite order in~$\Gamma$.  
Then  the Chen--Ruan orbifold cohomology vector space of $[Y/ \Gamma]$ is:
$$ \Homol^*_{\rm CR}([Y / \Gamma]) := \bigoplus_{g \in T} \Homol^* \left( Y^g / C_\Gamma(g); \thinspace \rationals \right).$$
The grading on this vector space is reviewed in Equation~\eqref{grading} below.
\end{df}
This definition is slightly different from, but equivalent to, the original one in \cite{ChenRuan}.
We can verify this fact using arguments analogous to those used by Fantechi and G\"ottsche \cite{FantechiGoettsche}
in the case of a finite group~$\Gamma$ acting on~$Y$.
The additional argument needed when considering some element $g$ in~$\Gamma$ of infinite order is the following. 
As the action of~$\Gamma$ on $Y$ is properly discontinuous, $g$ does not admit any fixed point in $Y$. Thus, 
$ \Homol^* \left( Y^g / C_\Gamma(g); \thinspace \rationals \right) = \Homol^* \left( \emptyset; \thinspace \rationals \right) = 0. $
For another proof, see  \cite{AdemLeidaRuan}.

\subsection*{Statement of the results}
We  first compute the Chen--Ruan Orbifold Cohomology for
 the complex orbifolds $[\Hy^3_\C / \Gamma]$ in the following way.
%
%
In order to describe the vector space structure of 
$ \Homol^*_{\rm CR}([\Hy^3_\C / \Gamma])$, we reduce
our considerations on complex orbifolds to the easier case of real orbifolds.
This is achieved using Theorem~\ref{spine}
(Section~\ref{A spine for the complexified Bianchi orbifolds}),
which says that there is a $\Gamma$-equivariant homotopy equivalence 
between $\HHH_\CC^3$ and $\HHH_\RR^3$. 

As a result of Theorems~\ref{3-torsion quotients} and~\ref{2-torsion quotients},
 we can express the vector space structure of the Chen-Ruan 
 orbifold cohomology
 in terms of the numbers of conjugacy classes of finite subgroups
and the cohomology of the quotient space.
Actually, Theorems~\ref{3-torsion quotients} and~\ref{2-torsion quotients}
hold true for
(finite index subgroups in) Bianchi groups with units~$\{\pm 1\}$.
These latter groups are the groups $\PSL_2(\ringO)$, where 
$\ringO$ is a ring of integers of imaginary quadratic number fields  
such that it admits as only units $\{\pm 1\}$. The remaining cases are 
the Gaussian and Eisenstein integers, and we treat them
separately in Sections~\ref{GaussExample} and~\ref{EisensteinExample},
respectively.

More precisely, as a corollary to Theorems~\ref{3-torsion quotients} and~\ref{2-torsion quotients}, 
which we are going to prove in Section~\ref{Kraemer numbers and orbifold cohomology}, 
and using Theorem~\ref{spine}, we obtain:

\begin{corollary} \label{introduced result}
Let $\Gamma$ be a finite index subgroup in a Bianchi group with units $\{\pm 1 \}$.
Denote by $\lambda_{2n}$ the number of conjugacy classes of cyclic subgroups of order ${n}$
in $\Gamma$.
Denote by $\lambda_{2n}^*$ the number of conjugacy classes of those of them which are contained in a dihedral subgroup of order $2n$
in~$\Gamma$. Then,
$$ \Homol^d_{\rm CR}\left([\Hy^3_\C / \Gamma] \right) \cong 
\Homol^d\left(\Hy^3_\R / {\Gamma}; \thinspace \rationals \right) \oplus
\begin{cases}
\rationals^{\lambda_4 +2\lambda_6 -\lambda_6^*}, & d=2, \\
 \rationals^{\lambda_4-\lambda_4^* +2\lambda_6 -\lambda_6^*}, & d=3, \\
0, & \mathrm{otherwise}. \end{cases}$$ 
\end{corollary}

Together with the example computations for the Gaussian and Eisensteinian cases (Sections~\ref{GaussExample} and~\ref{EisensteinExample}),
we obtain $ \Homol^d_{\rm CR}\left([\Hy^3_\C / \Gamma] \right)$ for all Bianchi groups $\Gamma$.

The (co)homology of the quotient space $\Hy^3_\R  / \Gamma$
 has been  computed numerically for a large scope of Bianchi groups \cite{Vogtmann}, \cite{Scheutzow}, \cite{Higher_torsion};
 and bounds for its Betti numbers have been given in \cite{KraemerThesis}.
Kr\"amer \cite{Kraemer} has determined number-theoretic formulae
 for the numbers $\lambda_{2n}$ and $\lambda_{2n}^*$ of conjugacy classes of finite subgroups in the Bianchi groups.
Kr\"amer's formulae have been evaluated for hundreds of thousands of Bianchi groups \cite{accessing_Farrell_cohomology}, 
and these values are matching with the ones from the orbifold structure computations with \cite{BianchiGP}
in the cases where the latter are available.

Using the previous description of $\Homol^*_{\rm CR}([\Hy^3_\C / \Gamma])$ and Theorem~\ref{obstructionBundles}
we can compute the Chen-Ruan cup product as follows. By degree reasons,
the Chen-Ruan cup product $\alpha_g \cup_{\rm CR} \beta_h$ between cohomology classes of two twisted sectors 
is zero. On the other hand, if $\alpha_g \in \Homol^* \left( (\Hy^3_\C )^g / C_\Gamma (g)\right)$ and $\beta \in \Homol^* \left( \Hy^3_\C /\Gamma \right)$,
then $\alpha_g \cup_{\rm CR} \beta = \alpha_g \cup \imath_g^* \beta \in \Homol^* \left( (\Hy^3_\C )^g / C_\Gamma (g)\right)$,
where $\imath_g \colon ( \Hy^3_\C )^g / C_\Gamma (g) \to \Hy^3_\C/\Gamma$ is the natural map 
induced by the inclusion $( \Hy^3_\C )^g \subset \Hy^3_\C$ (notice that in this case the obstruction bundle 
has fiber dimension zero by Theorem~\ref{obstructionBundles}). 

Let us  consider now, 
for any complexified Bianchi orbifold $[\Hy^3_\C /\Gamma]$, 
its coarse moduli space
$\Hy^3_\C / \Gamma$.  It is a quasi-projective variety (\cite{Baily-Borel}) with Gorenstein singularities (Lemma~\ref{stab in SL3}).
Therefore, it admits a crepant resolution (see e.g. \cite{Roan}, \cite{ChenTseng}).
Then we prove the following  result.
\begin{theorem}\label{mainthm}
Let $\Gamma$ be a Bianchi group  and let $[\Hy^3_\C / \Gamma]$ be the corresponding orbifold,
with coarse moduli space  $\Hy^3_\C / \Gamma$.
Let  $f\colon Y \ra \Hy^3_\C /\Gamma$ be a crepant resolution of $\Hy^3_\C /\Gamma$.
Then there  is an isomorphism as graded $\QQ$-algebras between the Chen--Ruan
cohomology ring of $[\Hy^3_\C / \Gamma]$ and the singular cohomology ring of $Y$:
$$
\left( H_{\rm CR}^*([\Hy^3_\C / \Gamma]) , \cup_{\rm CR} \right) \cong  
\left( H^*(Y) , \cup \right) \, .
$$
\end{theorem}
The proof of this theorem, which we shall give in Section~\ref{ringIsomorphism},
uses the McKay correspondence and our computations of the 
Chen--Ruan orbifold cohomology of the complexified Bianchi orbifolds.   
In Section~\ref{vanishing_section}, we compare this result with Ruan's Cohomological Crepant Resolution Conjecture
(\cite{RuanCCRC}, \cite{CoatesRuan}). Even though $\Hy^3_\C /\Gamma$ and $Y$ are not projective varieties, hence Ruan's conjecture
does not apply directly,  our results confirm the validity of this conjecture.

Finally, in 
Section~\ref{Sample orbifold cohomology computations for the Bianchi groups}, 
we illustrate our study with the computation of some explicit examples.

\subsection*{Acknowledgements} 
We would like to thank Martin Deraux for helpful and motivating discussions, which initiated the present work.
Warmest thanks go to Alessandro Chiodo and Yongbin Ruan for explanations on the cohomological crepant resolution conjecture,
to Jos\'e Bertin, Dmitry Kerner, Christian Lehn and \mbox{Sergei Yakovenko} for  answering questions on resolutions of singularities,
to John Ratcliffe for answers on orbifold topology, 
to Philippe Elbaz-Vincent, Graham Ellis, Stephen S. Gelbart, Thomas Schick and Gabor Wiese for support and encouragement of the second author during the years that this paper needed for maturing,
and to the referee of the previous version of the paper, for having motivated the research which led to a full proof of the conjecture in the cases under our study.

The first author was also supported  by 
PRIN 2015EYPTSB-PE1 "Geometria delle varieta' algebriche", 
FRA 2015 of the University of Trieste, the group GNSAGA of INDAM; the second one by Luxembourg's FNR grant INTER/DFG/FNR/12/10/COMFGREP.

\subsection*{Notation} We use the following symbols for the finite subgroups in $\mathrm{PSL}_2(\ringO)$:
for cyclic groups of order $n$, we write $\Z/n$;
for the Klein four-group \mbox{$\Z/2 \times \Z/2$}, we write $\Kleinfourgroup$;
for the dihedral group with six elements, we write $\Sthree$; and for the alternating group on four symbols, we write~$\Afour$. 

\section{The orbifold cohomology product} \label{The orbifold cohomology product}
In order to equip the orbifold cohomology vector space with a product structure,
called the Chen--Ruan product, we use the  orbifold complex  structure on $[Y / \Gamma]$. 
\\
Let $Y$ be a complex manifold of dimension $D$ with a properly discontinuous  action of a discrete group~$\Gamma$ by bi-holomorphisms.
For any $g \in \Gamma$ and $y \in Y^g$, we consider the eigenvalues $\lambda_1,\ldots,\lambda_D$ of the action of $g$ on the tangent space~$\mathrm{T}_{y}Y$.
As the action of $g$ on $\mathrm{T}_{y}Y$ is complex linear, its eigenvalues are roots of unity.
\begin{df}
Write $\lambda_j = e^{2\pi i r_j}$,
where $r_j$ is a rational number in the interval $[0,1[$. The \emph{degree shifting number} of $g$ in $y$ is the rational number $\age(g,y):=\sum_{j=1}^D r_j$.
\end{df}
The degree shifting number agrees with the original definition by Chen and Ruan
(see  \cite{FantechiGoettsche}). It is also called the fermionic shift number in \cite{Zaslow}.
The degree shifting number of an element $g$ is constant on a connected component of its fixed point set $Y^g$. 
For the groups under our consideration, $Y^g$ is connected,
so we can omit the argument $y$. 
Details for this and the explicit value of the degree shifting number are given in Lemma~\ref{rotationAge}.
Then we can define the graded vector space structure of the Chen--Ruan orbifold cohomology as
\begin{equation} \label{grading}
 \Homol^d_{\rm CR}([Y / \Gamma]) := \bigoplus_{g \in T} \Homol^{d-2 \thinspace \age(g)} \left( Y^g / C_\Gamma(g); \thinspace \rationals \right).
\end{equation}

Denote by $g$, $h$ two elements of finite order in~$\Gamma$, and by $Y^{g,h}$ their common fixed point set. Chen and Ruan construct a certain vector bundle on $Y^{g,h}$,
the \emph{obstruction bundle}. We denote by $c(g,h)$ its top Chern class.
In our cases, $Y^{g,h}$ is a connected manifold. In the general case, the fiber dimension of the obstruction bundle can vary between the connected components of $Y^{g,h}$, and $c(g,h)$ is the cohomology class restricting to the top Chern class of the obstruction bundle on each connected component.
The obstruction bundle is at the heart of the construction \cite{ChenRuan} of the Chen--Ruan orbifold cohomology product.
In \cite{FantechiGoettsche}, this product, when applied to a cohomology class associated to $Y^g$ and one associated to $Y^h$, is described as a push-forward of the cup product of these classes restricted to $Y^{g,h}$ and multiplied by $c(g,h)$.
\\
The following statement is made for global quotient orbifolds, but it is a local property, so we can apply it in our  case.
\begin{lemma}[Fantechi--G\"ottsche] \label{FantechiGoettscheFormula}
Let $Y^{g,h}$ be connected. Then the obstruction bundle on it is a vector bundle of fiber dimension
$$ \age(g) +\age(h) -\age(gh) -\codim\left(Y^{g,h} \subset Y^{gh}\right). $$
\end{lemma}
In \cite{FantechiGoettsche}, a proof is given in the more general setting that $Y^{g,h}$ needs not be connected. 
Examples where the product structure is worked out in the non-global quotient case,
 are for instance given in \cite{ChenRuan}*{5.3}, \cite{Perroni} and~\cite{BoissiereMannPerroni}.

\subsection{Groups of hyperbolic motions}
A class of examples with complex structures admitting the grading \eqref{grading} 
is given by the discrete subgroups~$\Gamma$ of the orientation preserving isometry group PSL$_2(\mathbb{C})$ of real hyperbolic 3-space $\mathcal{H}^3_\R$.
The Kleinian model  of $\mathcal{H}^3_\R$ gives a natural identification of the orientation preserving  isometries of $\mathcal{H}^3_\R$ with matrices in $\PSO(1,3)$.
By the subgroup inclusion $\PSO(1,3) \hookrightarrow \PSU(1,3)$,
these matrices specify isometries of the complex hyperbolic space $\Hy^3_\C$. The details are as follows.

\begin{construction} \label{complexification}
 Given an orbifold $[\Hy^3_\R / \Gamma]$, we presently construct 
 the complexified orbifold $[\Hy^3_\C / \Gamma]$.
 Recall the Kleinian model for $\mathcal{H}^3_\R$ described in~\cite{ElstrodtGrunewaldMennicke}:
 For this, we take a basis $\{f_0, f_1, f_2, f_3 \}$ for $\R^4$, and rewrite $\R^4$ as
 $\tilde{E}_1 := \R f_0 \oplus \R f_1 \oplus \R f_2 \oplus \R f_3$.
 Then we define the quadratic form $q_1$ by
 $$q_1(y_0f_0 + y_1f_1+ y_2f_2 +y_3f_3) = y_0^2 - y_1^2 - y_2^2 -y_3^2.$$
  We consider the real projective 3-space
  $\mathbb{P}\tilde{E}_1 = (\tilde{E}_1 \setminus \{0\})/\R^*,$
  where $\R^*$ stands for the multiplicative group $\R \setminus \{0\}$.
  The set underlying the Kleinian model is then
  $${\mathbb{K}} := \{ [y_0: y_1: y_2: y_3] \in \mathbb{P}\tilde{E}_1 \medspace | \medspace q_1(y_0, y_1, y_2, y_3) > 0 \}.$$
  Once that $\mathbb{K}$ is equipped with the hyperbolic metric, its group of orientation preserving isometries is PSO$_4(q_1, \R) =:$ PSO$(1,3)$.
  The isomorphism of $\mathbb{K}$ to the upper-half space model of $\mathcal{H}^3_\R$ yields an isomorphism between the groups of orientation preserving isometries, 
  PSO$(1,3) \cong$ PSL$_2(\C)$. This is how we include $\Gamma$ into PSO$(1,3)$.
  \\
  Now we consider the complex Euclidean 4-space $\tilde{E}_1 \otimes_\R \C := \C f_0 \oplus \C f_1 \oplus \C f_2 \oplus \C f_3,$
the  complex projective 3-space
  $\mathbb{P}(\tilde{E}_1 \otimes_\R \C ) = (\tilde{E}_1 \otimes_\R \C \setminus \{0\})/\C^*,$
  and obtain a model $${\mathbb{K}}_\C := \{ [z_0: z_1: z_2: z_3] \in \mathbb{P}(\tilde{E}_1 \otimes_\R \C ) \medspace | \medspace q_1(|z_0|, |z_1|, |z_2|, |z_3|) > 0 \}$$
  for the complex hyperbolic 3-space $\Hy_\C^3$, where $|z|$ denotes the modulus of the complex number $z$.
  The latter model admits PSU$(1,3)$ as its group of orientation preserving isometries, with a natural inclusion of PSO$(1,3)$.
  \\
  This is how we obtain our action of $\Gamma$ on $\Hy_\C^3$. In the remainder of this section we show some properties of 
  this action that will be used in the following.
\end{construction}

\vbox{
\begin{lemma} \label{ActionProperlyDiscontinuous}
The action of $\Gamma$ on $\Hy^3_\C$ just defined is properly discontinuous.  
\end{lemma}
\begin{proof}
This fact should be well known and can be proved using the existence of Dirichlet fundamental domains 
for the $\Gamma$-action on $\Hy^3_\C$ \cite{Goldman}*{Section 9.3}.
We include here for completeness a self-contained proof which relies on the  fact that the $\Gamma$-action on  $\Hy^3_\R$ is properly discontinuous
\cite[Theorems 1.2, p. 34, and 1.1 p. 311]{ElstrodtGrunewaldMennicke}.

Let $\{ \gamma_n\}_{n\geq 1}$ be a sequence of elements of $\Gamma$ and let 
$x\in \Hy^3_\C$ be a point, such that $\{ \gamma_n \cdot x \}_{n\geq 1}$ is infinite.
We show that $\{ \gamma_n \cdot x \}_{n\geq 1}$ has no accumulation point in $\Hy^3_\C$.
To this aim, assume by contradiction that 
$x_\infty \in \Hy^3_\C$ is an  accumulation point for $\{ \gamma_n \cdot x \}_{n\geq 1}$.
Let   $p\colon \Hy^3_\C \to \Hy^3_\R$ be the projection  defined in the proof of Theorem~\ref{spine} and consider
$p(x_\infty)$,  $\{ p(\gamma_n \cdot x) \}_{n\geq 1}$. Notice that since $p$ is $\Gamma$-equivariant, $p(\gamma_n \cdot x)= \gamma_n \cdot p(x)$,
and $\{ \gamma_n \cdot p(x) \}_{n\geq 1}$ is infinite, because $\Gamma$ acts properly discontinuously on  $\Hy^3_\R$.
It follows that $p(x_\infty)$ is an accumulation point for $\{ \gamma_n \cdot p(x)\}_{n\geq 1}$, hence a contradiction.
\end{proof}
}

\begin{lemma} \label{ProperTwistedSectors}
For any $g\in \Gamma$, the natural map 
$\left( \Hy^3_\C \right)^g / C_\Gamma (g) \to \Hy^3_\C/\Gamma$
induced by the inclusion $\left( \Hy^3_\C \right)^g \subset \Hy^3_\C$
is proper.
\end{lemma}
\begin{proof}
The proof is given in two steps, in the first one we show that the  map has finite fibres.
Since this fact holds true in general, for any discrete group $\Gamma$ acting properly discontinuously 
by bi-holomorphisms on a complex manifold $M$, we prove it in this generality. 
Let us denote by $f\colon M^g/C_\Gamma (g) \to M/\Gamma$ the natural map
induced by the inclusion $M^g \subset M$.
For any $x \in M^g$, let $[x] \in M^g/C_\Gamma (g)$
be its equivalence class. Then
$$
f^{-1}(f([x])) = \{ y \in M^g \, | \, y\in \Gamma \cdot x  \}/C_\Gamma (g) \, ,
$$   
where $\Gamma \cdot x$ denotes the orbit of $x$. Notice that for any $h\in \Gamma$, if $h\cdot x \in M^g$,
then $g\in {\rm Stab}(h\cdot x) = h{\rm Stab}(x)h^{-1}$, and so there exists a unique $g_h \in {\rm Stab}(x)$
such that $hg_h h^{-1} =g$, $g_h =  h^{-1}  g h$. Here for any element $y$,  ${\rm Stab}(y)$ denotes its stabilizer.
Furthermore, if $h_1 , h_2 \in \Gamma$ are such that
$h_1^{-1}  g h_1 = h^{-1}_2  g h_2$, then $g=h_2h_1^{-1} g h_1h_2^{-1} = (h_2h_1^{-1} )g (h_2h_1^{-1} )^{-1}$.
Therefore, $h_2h_1^{-1} \in C_\Gamma (g)$ and hence $h_2 \in C_\Gamma (g) \cdot h_1$.
This implies that, if we define
$$
\Gamma_{x,g} := \{ h\in \Gamma \, | \, h\cdot x \in M^g \} \, ,
$$
then the map $f_{x,g} \colon \Gamma_{x,g} \to {\rm Stab}(x)$, $h \mapsto g_h = h^{-1}g h$,
descends to an injective map $\Gamma_{x,g} / C_\Gamma (g) \to {\rm Stab}(x)$.
The claim now follows from the fact that ${\rm Stab}(x)$ is finite and $\Gamma_{x,g} / C_\Gamma (g)$
is bijective to $f^{-1}(f([x]))$.

In the second step of the proof, $M= \Hy^3_\C $ and $f\colon \left( \Hy^3_\C \right)^g / C_\Gamma (g) \to \Hy^3_\C/\Gamma$.
Let $d\colon \Hy^3_\C \times \Hy^3_\C \to \mathbb{R}$ be the distance function
induced by the Bergman metric, that is the positive definite Hermitian form 
$\sum_{\alpha , {\beta}}^3 \frac{\partial^2 \log \mathcal{K}}{\partial z_\alpha \partial \bar{z}_{\beta}}
{\rm d}z_\alpha {\rm d} \bar{z}_\beta$ on $\Hy^3_\C$,
where $\mathcal{K}$ is the Bergman kernel of $\Hy^3_\C$
(see \cite[p. 145]{MorrowKodaira}). By restriction 
$d$ induces a distance function on $\left( \Hy^3_\C \right)^g$. Moreover, defining for any $[x], [y] \in \Hy^3_\C/\Gamma$
(respectively $[x], [y] \in \left( \Hy^3_\C \right)^g / C_\Gamma (g)$),
$$
\tilde{d} ([x], [y] ) := {\rm Inf} \{ d(\xi , \eta ) \, | \, \xi \in \Gamma \cdot x \, , \, \eta \in \Gamma \cdot y \} \, ,
$$
we have a distance function on $\Hy^3_\C/\Gamma$ (on $\left( \Hy^3_\C \right)^g / C_\Gamma (g)$ respectively,
where $\tilde{d}$ is defined accordingly).
By elementary topology, for metric spaces, a subspace $K$ is compact
if and only if  any infinite subset $Z\subset K$ has an accumulation point in $K$. 
So, let $K \subset \Hy^3_\C/\Gamma$ be a compact subspace. To show that $f^{-1}(K)$ is  compact,
let $Z \subset f^{-1}(K)$ be an infinite subset. Since $f$ has finite fibres, $f(Z)$ is infinite, so it has an
accumulation point, say $[x_0] \in K$.  
Notice that $f^{-1}([x_0]) \not= \emptyset$, since ${\rm Im}(f)$ is closed.
To see this, let $[x] \not\in {\rm Im}(f)$. Then $\Gamma \cdot x \cap \left( \Hy^3_\C \right)^g = \emptyset$,
in other words, for any $y \in \Gamma \cdot x$, $g\not\in {\rm Stab}(y)$. Since the action is properly discontinuous,
any $y \in \Gamma \cdot x$ has a neighborhood $U$ such that $\gamma \cdot U \cap U \not= \emptyset$,
if and only if $\gamma \in {\rm Stab}(y)$, for any $\gamma \in \Gamma$. In particular, the stabilizer of any 
point in  $U$ is contained in ${\rm Stab}(y)$, and hence $\Gamma \cdot U \cap \left( \Hy^3_\C \right)^g = \emptyset$. So, $\Gamma \cdot U$ gives an open neighbourhood of $[x]$
which has empty intersection with ${\rm Im}(f)$.
To finish the proof of the lemma, we observe that,
if $[x_0]\in K$ is an accumulation point for $f(Z)$, and
$f^{-1}([x_0]) \not= \emptyset$, then there exists 
$[y_0] \in f^{-1}([x_0]) \subset f^{-1}(K)$ which is an accumulation point for $Z$,
since $f$ has finite fibres.
\end{proof}

As we will see in Section \ref{A spine for the complexified Bianchi orbifolds}
(Remark \ref{rem spine}), if $g\in {\rm PSL}_2(\C) \cong {\rm PSO}(1,3)$
is different from $\pm 1$, and $(\Hy^3_\C)^g \not= \emptyset$, then 
$(\Hy^3_\C)^g \cap \Hy^3_\R \not= \emptyset$. Therefore, 
$g$ is an elliptic element of ${\rm PSL}_2(\C)$ 
(\cite[Prop. 1.4, p. 34]{ElstrodtGrunewaldMennicke}), in particular it has 
exactly two fixed points on $\partial \Hy^3_\R  \cong \PP^1_\C$,
and the geodesic line in $\Hy^3_\R$ joining these two points is contained in 
 $(\Hy^3_\C)^g$. Moreover, $g$ acts as a rotation around this geodesic line.
For this reason, we call any such element $g$ a \textbf{non-trivial rotation}
of $\Hy^3_\C$.

\vbox{
\begin{lemma} \label{rotationAge}
The degree shifting number of any non-trivial rotation of $\Hy^3_\C$ on its fixed points set is $1$.  
\end{lemma}
\begin{proof}
For any rotation $\hat{\theta}$ of angle $\theta$ around a geodesic line in $\mathcal{H}^3_\R$, there is a basis for the construction of the Kleinian model
 such that the matrix of $\hat{\theta}$ takes the following shape (\cite[Prop. 1.13, p. 40]{ElstrodtGrunewaldMennicke}),
$$\left( \begin{array}{cccc}
1 & 0 & 0 & 0 \\
0 & 1 & 0 & 0 \\
0 & 0 & \cos \theta & -\sin \theta  \\
0 & 0 & \sin \theta & \cos \theta 
  \end{array} \right)
\in \PSO(1,3).$$
This matrix, considered as an element of $\PSU(1,3)$, 
performs a rotation of angle $\theta$ around the ``complexified geodesic line'' with respect to the inclusion $\mathcal{H}^3_\R \hookrightarrow \mathcal{H}^3_\C$.
The fixed points of this rotation are exactly the points $x$ lying on this complexified geodesic line, 
and the action on their tangent space T$_x\Hy^3_\C \cong \C^3$ is again a rotation of angle $\theta$.
Hence we can choose a basis of this tangent space such that this rotation is expressed by the matrix
$$\left( \begin{array}{ccc}
1 & 0 & 0\\
0&e^{i\theta} & 0  \\
0& 0 & e^{-i\theta} 
  \end{array} \right)
\in {\rm SL}_3(\C).$$
Therefore, the degree shifting number of the rotation $\hat{\theta}$ at $x$ is 1.

Let now $x \in (\Hy^3_\C)^g \setminus \Hy^3_\R$. From Remark \ref{rem spine} it follows that 
$x$ and $p ( x ) \in (\Hy^3_\R)^g$ belongs to the same connected component
of $(\Hy^3_\C)^g$, where $p\colon \Hy^3_\C \to \Hy^3_\R$ is the projection defined in the proof of Theorem~\ref{spine}.
Therefore,  $\age(g,x) = \age(g,p( x )) =1$.
\end{proof}
}

\begin{lemma}\label{stab in SL3}
Let $\Gamma$ be a Bianchi group  acting on $\Hy^3_\C$ as in Construction
\ref{complexification}. Then, for any point $x\in \Hy^3_\C$, 
the stabilizer ${\rm Stab}_\Gamma (x)$
of $x$ in $\Gamma$ is a finite group isomorphic to one of the following groups:
the cyclic group of order $1, 2$ or $3$;  the  dihedral group $\Kleinfourgroup$
 of order $4$;  the dihedral group $\Sthree$ of order $6$;
the alternating group  $\Afour$ of order $12$.

Furthermore,  the map ${\rm Stab}_\Gamma ( x ) \to GL_3(\C)$ given by 
$\gamma \mapsto T_x \gamma$, where $T_x \gamma$ is the differential of 
$\gamma$ at $x$,
is an injective group homomorphism, whose image is contained in $SL_3(\C)$
and it is conjugate to one of the following subgroups $G$ of $SL_3(\C)$:
\begin{itemize}
\item[1)] if ${\rm Stab}_\Gamma (x)$ is cyclic of order $n=1, 2$ or $3$, then 
$G= \left\langle \begin{pmatrix}\omega & 0 & 0 \\ 0 &\omega^{-1} & 0 \\ 0 & 0& 1\end{pmatrix}\right\rangle$,
where $\omega \in \C^*$ is a primitive $n$-th root of $1$.
\item[2)] If ${\rm Stab}_\Gamma (x) \cong \Kleinfourgroup$, 
then $G= \left\langle \begin{pmatrix} -1 & 0 & 0 \\ 0 & -1 & 0 \\ 0 & 0 & 1 \end{pmatrix},
\begin{pmatrix} -1 & 0 & 0 \\ 0 & 1 & 0 \\ 0 & 0 & -1 \end{pmatrix} \right\rangle$.
\item[3)] If ${\rm Stab}_\Gamma (x) \cong \Sthree$, then 
$G= \left\langle \begin{pmatrix} \omega & 0 & 0 \\ 0 & \omega^2 & 0 \\ 0 & 0 & 1 \end{pmatrix},
\begin{pmatrix} 0 & 1 & 0 \\ 1 & 0 & 0 \\ 0 & 0 & -1 \end{pmatrix} \right\rangle$\, ,
where $\omega \in \C^*$ is a primitive third root of $1$.
\item[4)] If ${\rm Stab}_\Gamma (x) \cong \Afour$, then 
$G= \left\langle \begin{pmatrix} -1 & 0 & 0 \\ 0 & -1 & 0 \\ 0 & 0 & 1 \end{pmatrix},
\begin{pmatrix} -1 & 0 & 0 \\ 0 & 1 & 0 \\ 0 & 0 & -1 \end{pmatrix}, 
\begin{pmatrix}0&1&0\\0&0&1\\1&0&0\end{pmatrix}\right\rangle$.
\end{itemize}
\end{lemma}
\begin{proof}
Since the action of $\Gamma$ on $\Hy^3_\C$ is properly discontinuous  
(Lemma~\ref{ActionProperlyDiscontinuous}),  ${\rm Stab}_\Gamma ( x )$ is finite.
The first part of the lemma follows now from the 
 classification of the finite subgroups of $\Gamma$
(see Lemma~\ref{finiteSubgroups}).

From the proof of Lemma~\ref{rotationAge} we deduce that, if $\gamma 
\in {\rm Stab}_\Gamma ( x ) \setminus \{ 1 \}$, then $T_x \gamma$ is different from the 
identity and  $\det ( T_x \gamma ) =1$, hence we obtain an injective group homomorphism
${\rm Stab}_\Gamma ( x ) \to SL_3 ( \C)$.
The description of the images of these morphisms follows from elementary 
representation theory, as we briefly explain. 

The case 1) is clear.
In case 2), $G$ is generated by two matrices $A, B \in SL_3(\C)$, such that
$A^2=B^2=I_3$, and $A\cdot B = B\cdot A$. From Schur's lemma it follows 
that $A$ and $B$ are simultaneously diagonalisable, hence there exists a basis of $\C^3$
such that $A$ and $B$ are diagonal of the given form. 

In case 3), $G$ is generated by two matrices, $A, B$, such that 
$A^3=B^2=(A\cdot B)^2=I_3$. Let $\{ u,v,w\}$ be a basis of $\C^3$, such that
$Au=\omega u$,   $Av=\omega^2 v$,  $Aw=w$, where $\omega \in \C^*$, $\omega^3 =1$, 
$\omega \not= 1$. From the relation $A\cdot B=B\cdot A^2$, we deduce that
$Bw=\pm w$, $Bu=av$ and $Bv=bu$, for some $a,b \in \C^*$. Since $B^2=I_3$,
it follows that $ab=1$. Hence, in the basis $\{ \frac{1}{a}u,v,w\}$ of $\C^3$,
the matrices $A$ and $B$ have the desired form.

Finally, in case 4), we use the fact that $\Afour$ has four irreducible representations 
(see e.g. \cite[Thm. 7, p. 19]{Serre}), three  of dimension one that are induced 
by the representations of $\Afour / H \cong \ZZ/3\ZZ$, where $H$ is the normal subgroup
of $\Afour$ consisting of the permutations of order two. The remaining 
irreducible representation of $\Afour$ is of dimension three. Therefore,
up to conjugation, there is only one injective group homomorphism 
$\Afour \to SL_3(\C)$. The result follows  from the fact that the three given 
matrices generate a subgroup of $SL_3(\C)$ isomorphic to $\Afour$.
\end{proof}

\vbox{
\begin{thm} \label{obstructionBundles}
 Let~$\Gamma$ be a group generated by translations and rotations of $\Hy^3_\C$. Then all obstruction bundles of the orbifold $[\Hy^3_\C / \Gamma]$ 
 are of fiber dimension zero, except in the case corresponding to two elements $b, c\in \Gamma \setminus \{ 1\}$ with $c=b^k$ and $b c \not=1$.
In this last case the obstruction bundle   is of fiber dimension $1$ and it is trivial.
\end{thm}
\begin{proof}
Non-trivial obstruction bundles can only appear for two elements of~$\Gamma$ with common fixed points.
The translations of $\Hy^3_\C$ have their fixed points on the boundary and not in $\Hy^3_\C$.
So let $b$ and $c$ be non-trivial hyperbolic rotations around distinct axes intersecting in the point $x \in \Hy^3_\C =:Y$. Then $bc$ is again a hyperbolic rotation around a third distinct axis passing through $x$. 
Obviously, these rotation axes constitute the fixed point sets $Y^b$, $Y^c$ and $Y^{bc}$. Hence the only fixed point of the group generated by~$b$ and~$c$ is~$x$.
Now Lemma~\ref{FantechiGoettscheFormula} yields the following fiber dimension for the obstruction bundle on $Y^{b,c}$:
$$ {\age}(b) + {\age}(c) -{\age}(bc) -\codim\left(Y^{b,c} \subset Y^{bc} \right).$$ 
After computing degree shifting numbers using Lemma~\ref{rotationAge}, we see that this fiber dimension is zero.

Let now $b$ and $c$ be non-trivial hyperbolic rotations around the same axis $Y^b=Y^c$. Then $c=b^k$ and either  $bc=1$, or $bc\not=1$.
As before we conclude  that   the fiber dimension of the obstruction bundle is $0$ in the first case, and $1$ in the second.
However, if $bc\not= 1$, then $Y^{b,c} = Y^{bc}$, which is a non-compact Riemann surface contained in $\Hy^3_\C$, hence
the obstruction bundle is trivial in this case \cite[Thm. 30.3, p. 229]{Forster}. 

Finally, if $b=1$, or $c=1$, the claim follows from  Lemma~\ref{FantechiGoettscheFormula} and Lemma~\ref{rotationAge} as before.
\end{proof}
}

\section{The centralizers of finite cyclic subgroups in the Bianchi groups}
 \label{The conjugacy classes of finite order elements in the Bianchi groups}

In this section, as well as Theorems~\ref{3-torsion quotients} and~\ref{2-torsion quotients},
 we will reduce all our considerations to the action on real hyperbolic 3-space $\mathcal{H}^3_\R$.
 For the latter action, there are Poincar\'e's formulas~\cite{Poincare} on the upper-half space model, which extend the M\"obius transformations from the hyperbolic plane. 
Let $\Gamma$ be a finite index subgroup in a Bianchi group $\text{PSL}_2(\mathcal{O}_{-m})$.
In 1892, Luigi \mbox{Bianchi}~\cite{Bianchi} computed fundamental domains for some of the full Bianchi groups.
Such a fundamental domain has the shape of a hyperbolic polyhedron (up to a missing vertex at certain cusps,
 which represent the ideal classes of $\ringO_{-m}$), so we will call it the \emph{\mbox{Bianchi} fundamental polyhedron}.
We use the \mbox{Bianchi} fundamental polyhedron to induce a $\Gamma$-equivariant cell structure on $\mathcal{H}^3_\R$,
namely we start with this polyhedron as a $3$-cell, record its polyhedral facets, edges and vertices, and tessellate out $\mathcal{H}^3_\R$ with their $\Gamma$-images.

It is well-known \cite{binaereFormenMathAnn9} (cf. also \cite[Prop. 1.13, p. 40]{ElstrodtGrunewaldMennicke})
 that any element of~$\Gamma$ fixing a point inside real hyperbolic 3-space $\mathcal{H}^3_\R$ acts as a rotation of finite order. 
And the rotation axis does not pass through the interior of the Bianchi fundamental polyhedron, because the interior of the latter contains only one point on each $\Gamma$-orbit.
Therefore, we can easily refine our $\Gamma$-equivariant cell structure such that the stabilizer in~$\Gamma$ of any cell $\sigma$ fixes $\sigma$ point-wise:
We just have to subdivide the facets and edges of the Bianchi fundamental polyhedron by their symmetries 
(and then again spread out the subdivided cell structure on $\mathcal{H}^3_\R$ using the $\Gamma$-action).
This has been implemented in practice~\cite{Rahm_homological_torsion}, and we shall denote $\mathcal{H}^3_\R$ with this refined cell structure by $Z$.
 
\begin{df}
Let~$\ell$ be a prime number. The \emph{$\ell$--torsion sub-complex} is the sub-complex of $Z$ consisting of all the cells
 which have stabilizers in~$\Gamma$ containing elements of order $\ell$.
\end{df}
For $\ell$ being one of the two occurring primes $2$ and $3$, the orbit space of this sub-complex is a finite graph,
 because the cells of dimension greater \mbox{than 1} are trivially stabilized in the refined cellular complex.
We reduce this sub-complex with the following procedure, motivated in~\cite{accessing_Farrell_cohomology}.

\begin{ConditionA} \label{cell condition}
In the $\ell$--torsion sub-complex, let $\sigma$ be a cell of dimension $n-1$, lying in the boundary of precisely two $n$--cells $\tau_1$ and~$\tau_2$,
 the latter cells representing two different orbits.
Assume further that no higher-dimensional cells of the $\ell$--torsion sub-complex touch $\sigma$;
and that the $n$--cell stabilizers admit an isomorphism
$\Gamma_{\tau_1} \cong \Gamma_{\tau_2}$. 
\end{ConditionA}

Where this condition is fulfilled in the $\ell$--torsion sub-complex,
 we merge the cells $\tau_1$ and $\tau_2$ along~$\sigma$ and do so for their entire orbits,
 if and only if they meet the following additional condition. 
We never merge two cells the interior of which contains two points on the same orbit.

\begin{ConditionB} 
The inclusion $ \Gamma_{\tau_1} \subset \Gamma_\sigma$ 
induces an isomorphism on group homology with $\Z/\ell$--coefficients under the trivial action.
\end{ConditionB}

\emph{The reduced $\ell$--torsion sub-complex is the $\Gamma$--complex obtained by orbit-wise merging two $n$--cells of the
 $\ell$--torsion sub-complex satisfying conditions~$\cellCondition$ and~$B$.}

We use the following classification of Felix Klein \cite{binaereFormenMathAnn9}.

\begin{Lem}[Klein] \label{finiteSubgroups}
The finite subgroups in $\mathrm{PSL}_2(\ringO)$ are exclusively of isomorphism types the cyclic groups of orders $1$, $2$ and $3$,
the dihedral groups $\Kleinfourgroup$ and $\Sthree$ (isomorphic to the Klein four-group \mbox{$\Z/2 \times \Z/2$},
 respectively to the symmetric group on three symbols) and the alternating group~$\Afour$. 
\end{Lem}

Now we investigate the associated normalizer groups. Straight-forward verification using the multiplication tables of the implied finite groups yields the following.
\begin{lemma} \label{normalizer types}
Let $G$ be a finite subgroup of ${\rm PSL}_2(\ringO_{-m})$. Then the type of the normalizer of any subgroup of type~$\Z/ \ell$ in $G$ is given as follows for $\ell = 2$ and $\ell = 3$, where we print only cases with existing subgroup of type $\Z/\ell$.
$$ \begin{array}{c|cccccc}
{\rm Isomorphism}\medspace { \rm type}\medspace {\rm of } \medspace G & \{1\} & \Z/2 & \Z/3 & \Kleinfourgroup & \Sthree & \Afour \\ 
\hline &&&&&&  \\
\mathrm{ normalizer } \medspace \mathrm{ of } \medspace \Z/2 &  & 
 \Z/2 &      & \Kleinfourgroup & \Z/2    & \Kleinfourgroup \\ 
\mathrm{ normalizer } \medspace \mathrm{ of } \medspace \Z/3 &  &
      & \Z/3 &                 & \Sthree & \Z/3.
\end{array} $$ \normalsize
\end{lemma}



\begin{lemma} \label{reflection}
Let $v \in \mathcal{H}^3_\R$ be a vertex with stabilizer in~$\Gamma$ of type $\Kleinfourgroup$ or $\Afour$.
Let $\gamma$ in~$\Gamma$ be a rotation of order $2$ around an edge $e$ adjacent to $v$.
Then the centralizer  $C_\Gamma(\gamma)$ reflects $\Hy^\gamma$ \mbox{--- which is the geodesic line through $e$ ---} onto itself at $v$.
\end{lemma}
\begin{proof}
Denote by $\Gamma_v$ the stabilizer of the vertex $v$.
In the case that $\Gamma_v$ is of type $\Kleinfourgroup$, which is Abelian,
it admits two order-2-elements $\beta, \beta\cdot\gamma$ centralizing~$\gamma$ and turning the geodesic line through~$e$
onto itself such that the image of~$e$ touches~$v$ from the side opposite to~$e$ (illustration:\includegraphics{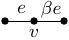}).
In the case that~$\Gamma_v$ is of type~$\Afour$, it contains a normal subgroup of type $\Kleinfourgroup$ that admits again two such elements.
\end{proof}

Any edge of the reduced torsion sub-complex is obtained by merging a chain of edges on the intersection of one geodesic line with
 some strict fundamental domain for~$\Gamma$ in $\Hy$.
We call this chain the \emph{chain of edges associated to $\alpha$}.
It is well defined up to translation along the rotation axis of~$\alpha$. 

\begin{lemma} \label{2-torsion axis}
Let $\alpha$ be any $2$--torsion element in~$\Gamma$.
Then the chain of edges associated to $\alpha$ is a fundamental domain for the action of the centralizer of $\alpha$
 on the rotation axis of $\alpha$.
\end{lemma}
\begin{proof}
We distinguish the following two cases of how $\left\langle \alpha \right\rangle \cong \Z/2$ is included into~$\Gamma$.

\textit{First case}.
Suppose that there is no subgroup of type $\Kleinfourgroup$ in~$\Gamma$ which contains $\left\langle \alpha \right\rangle$.
Then the connected component to which the rotation axis of $\alpha$ passes in the quotient of the  2-torsion subcomplex,
is homeomorphic to a circle, with cell structure $\circlegraph$.
We can write $\Gamma_e = \left\langle \alpha \right\rangle$ and $\Gamma_{e'} = \left\langle \gamma \alpha \gamma^{-1} \right\rangle$.
One immediately checks that any fixed point $x \in \Hy$ of $\alpha$ induces the fixed point $\gamma \cdot x$
 of $\gamma \alpha \gamma^{-1}$. As PSL$_2(\C)$ acts by isometries, the whole fixed point set in $\Hy$
 of $\alpha$ is hence identified by~$\gamma$ with the fixed point set of $\gamma \alpha \gamma^{-1}$.
This gives us the identification $\gamma^{-1}$ from $e'$ to an edge on the rotation axis of $\alpha$, adjacent to $e$ because of the first condition on $\gamma$.
We repeat this step until we have attached an edge $\delta e$ on the orbit of the first edge $e$, with $\delta \in \Gamma$.
As $\delta$ is an isometry, the whole chain is translated by $\delta$ from the start at $e$ to the start at $\delta e$.
So the group $ \left\langle \delta \right\rangle$ acts on the rotation axis with fundamental domain our chain of edges.
And $\delta \alpha \delta^{-1}$ is again the rotation of order $2$ around the axis of $\alpha$.
So, $ \delta \alpha \delta^{-1} = \alpha$ and therefore $ \left\langle \delta \right\rangle < C_\Gamma(\alpha)$.

\textit{Second case}.
Suppose that there is a subgroup $G$ of~$\Gamma$ of type $G \cong \Kleinfourgroup$ containing $\left\langle \alpha \right\rangle$.
Then the rotation axis of $\alpha$ passes in the quotient of the  2-torsion subcomplex to an edge on a connected component of homeomorphism type 
$\edgegraph$, $\graphFive$ or $\graphTwo$ (see~\cite{BerkoveRahm}).
If there is no further inclusion $G<G'<\Gamma$ with $G'\cong \Afour$, let $G' := G$.
Then the chain associated to $\alpha$ can be chosen such that one of its endpoints is stabilized by $G'$.
The other endpoint of this chain must then lie on a different $\Gamma$--orbit, and admit as stabilizer a group $H'$ containing $\left\langle \alpha \right\rangle$, of type $\Kleinfourgroup$ or $\Afour$.
By Lemma~\ref{reflection}, each $G'$ and $H'$ contain a reflection of the rotation axis of $\alpha$, centralizing $\alpha$. These two reflections must differ from one another because they do not fix the chain of edges.
So their free product tessellates the rotation axis of $\alpha$ with images of the chain of edges associated to $\alpha$.
\end{proof}

\section{A spine for the complexified Bianchi orbifolds} \label{A spine for the complexified Bianchi orbifolds}
In this section, we prove the following theorem, which will be used to prove Theorem
\ref{mainthm}.

\begin{theorem} \label{spine}
Let $\Gamma$ be a Bianchi group. Then there is a $\Gamma$-equivariant homotopy equivalence 
between $\HHH_\CC^3$ and $\HHH_\RR^3$. In particular, $\HHH_\CC^3/\Gamma$
is homotopy equivalent to $\HHH_\RR^3/\Gamma$.
\end{theorem}
\begin{proof}
We consider the ball model for  complex hyperbolic $3$-space $\HHH_\CC^3$ \cite{Goldman}
(which is called the Klein model in \cite{ElstrodtGrunewaldMennicke}). This provides us with a complex structure such that
$\HHH_\RR^3$ is naturally embedded into $\HHH_\CC^3$ as the fixed points set of the complex conjugation.
In the other direction, following \cite{Goldman}, we define a projection as follows. For any point $z\in \HHH_\CC^3$, 
there is a unique geodesic arc, with respect to the Bergman metric, $\alpha_{z, \bar{z}}$ from $z$ to its complex conjugate $\bar{z}$
(see e.g. \cite{Goldman}*{Theorem 3.1.11}); and the intersection point $p(z)= \alpha_{z, \bar{z}} \cap \HHH_\RR^3$
is equidistant to $z$ and $\bar{z}$ \cite{Goldman}*{Section 3.3.6}. This defines a projection 
$p\colon \HHH_\CC^3 \to \HHH_\RR^3$. Notice that $p$ is ${\rm PSO}(1,3)$-equivariant  and hence also $\Gamma$-equivariant.

Clearly, the restriction $p_{|\HHH_\RR^3}$ is the identity. On the other hand, let 
$$
H \colon \HHH_\CC^3 \times [0,1] \to \HHH_\CC^3 \, , \qquad H(z,t) = \alpha_{z, \bar{z}} \left( t\rho (z, p(z)) \right)
$$ 
where $\rho$ is the hyperbolic distance and we have parametrized the geodesic arc such that 
$\alpha_{z, \bar{z}}(0) = p(z)$ and $\alpha_{z, \bar{z}} ( \rho (z, p(z)) ) = z$. Then $H$
is an homotopy between $p$ and the identity map of $\HHH_\CC^3$. 
Furthermore, since ${\rm PSO}(1,3)$ is a group of  isometries of $\HHH_\CC^3$,
it sends geodesics to geodesics and so, for any $M\in {\rm PSO}(1,3)$,
\begin{eqnarray}\label{H equivariant}
H(Mz,t) &=&\alpha_{Mz, \overline{Mz}} \left( t\rho (Mz, p(Mz)) \right) \nonumber \\
&=& M\alpha_{z, \bar{z}} \left( t\rho (z, p(z)) \right) =MH(z,t) \, .
\end{eqnarray}
It follows that $H$ is   ${\rm PSO}(1,3)$-equivariant,
in particular it is $\Gamma$-equivariant.
\end{proof}

\begin{remark}\label{rem spine} \normalfont
From \eqref{H equivariant} it follows that, if $g\in {\rm PSO}(1,3)$
fixes a point $z$, with $z\not= \bar{z}$, then g fixes the geodesic arc $\alpha_{z, \bar{z}}$
pointwise. Indeed, from the fact that $g\cdot \bar{z} = \overline{g\cdot z} = \bar{z}$,
we deduce that $g(\alpha_{z, \bar{z}}) = \alpha_{z, \bar{z}}$. 
Moreover, for every $z' \in \alpha_{z, \bar{z}}$, we see that  
$g\cdot z' = z'$, because otherwise we get a contradiction from the following equalities:
$$
\rho (z,z') = \rho (g\cdot z , g\cdot z') = \rho (z, g\cdot z') \, ,
$$
where $\rho$ is the hyperbolic distance. 
\end{remark}

\begin{remark}\label{rem spine 2} \normalfont
From Lemma~\ref{stab in SL3} it follows that the points $z\in \Hy^3_\C$
such that the stabilizer  ${\rm Stab}_\Gamma (z ) \subset \Gamma$ is not cyclic
are isolated, hence,   Theorem~\ref{spine} implies that 
such points $z$ belong to $\Hy^3_\R$.

\end{remark}

\section{Orbifold cohomology of real Bianchi orbifolds} \label{Kraemer numbers and orbifold cohomology}

Our main results on the vector space structure of the Chen--Ruan orbifold cohomology of Bianchi orbifolds are the below two theorems.

\begin{theorem} \label{3-torsion quotients}
For any element $\gamma$ of order $3$ in a finite index subgroup~$\Gamma$ in a Bianchi group with units~$\{\pm 1\}$,
 the quotient space $\Hy^\gamma /_{C_\Gamma(\gamma)}$ of the rotation axis modulo the centralizer of $\gamma$
 is homeomorphic to a circle.
\end{theorem}
\begin{proof}
As $\gamma$ is a  non-trivial torsion element,  
by \cite{accessing_Farrell_cohomology}*{lemma 22}
 the~$\Gamma$--image of the chain of edges associated to $\gamma$ contains the rotation axis $\Hy^\gamma$.
  Now we can observe two cases.
\begin{itemize}
\item[$\circlegraph$]
First, assume that the rotation axis of $\gamma$ does not contain any vertex of stabilizer type~$\Sthree$ 
(from~\cite{accessing_Farrell_cohomology}, we know that this gives us a circle as a path component in the quotient of the 
 $3$--torsion sub-complex).
Assume that there exists a reflection of~$\Hy^\gamma$ onto itself by an element of~$\Gamma$.
Such a reflection would fix a point on~$\Hy^\gamma$.
Then the normalizer of $\left\langle \gamma \right\rangle$ in the stabilizer of this point would contain the reflection.
This way, Lemma~\ref{normalizer types} yields that this stabilizer is of type $\Sthree$, which we have excluded. 
Thus, there can be no reflection of~$\Hy^\gamma$ onto itself by an element of~$\Gamma$.
\\
As~$\Gamma$ acts by isometries preserving a metric of non-positive curvature (a CAT(0) metric), every element $g \in \Gamma$ sending an edge of the chain for~$\gamma$
 to an edge on~$\Hy^\gamma$ outside the fundamental domain, can then only perform a translation on~$\Hy^\gamma$.
A translation along the rotation axis of~$\gamma$ commutes with~$\gamma$, so $g \in {C_\Gamma(\gamma)}$.
Hence the quotient space $\Hy^\gamma /_{C_\Gamma(\gamma)}$ is homeomorphic to a circle.
\item[$\edgegraph$]
If~$\Hy^\gamma$ contains a point with stabilizer in~$\Gamma$ of type $\Sthree$, then there are exactly two $\Gamma$--orbits of such points.
The elements of order $2$ do not commute with the elements of order $3$ in $\Sthree$, so the centralizer of $\gamma$ does not contain the former ones.
Hence, ${C_\Gamma(\gamma)}$ does not contain any reflection of~$\Hy^\gamma$ onto itself.
Denote by $\alpha$ and $\beta$ elements of order $2$ of each of the stabilizers of the two endpoints of a chain of edges for~$\gamma$.
Then~$\alpha\beta$ performs a translation on~$\Hy^\gamma$ and hence commutes with $\gamma$.
A fundamental domain for the action of $\left\langle \alpha\beta \right\rangle$ on~$\Hy^\gamma$ is given by the chain of edges for~$\gamma$ united with its reflection through one of its endpoints.
As no such reflection belongs to the centralizer of $\gamma$ and the latter endpoint is the only one on its $\Gamma$--orbit in this fundamental domain, the quotient $\Hy^\gamma /_{C_\Gamma(\gamma)}$ matches with the quotient $\Hy^\gamma /_{\left\langle \alpha\beta \right\rangle}$, which is homeomorphic to a circle.
\end{itemize}
\end{proof}

\begin{theorem} \label{2-torsion quotients}
Let $\gamma$ be an element of order $2$ in a Bianchi group~$\Gamma$ with units~$\{\pm 1\}$.
Then, the homeomorphism type of the quotient space $\Hy^\gamma /_{C_\Gamma(\gamma)}$ is
\end{theorem}
\begin{itemize}
\item[$\edgegraph$]
an edge without identifications, if $\left\langle \gamma \right\rangle$ is contained in a subgroup of type $\Kleinfourgroup$
inside~$\Gamma$ and 
\item[$\circlegraph$]
a circle, otherwise.
\end{itemize}
\begin{proof}
By Lemma~\ref{2-torsion axis},
the chain of edges for~$\gamma$ is a fundamental domain for~${C_\Gamma(\gamma)}$
 on the rotation axis $\Hy^\gamma$ of $\gamma$.
 Again, we have two cases.
\begin{itemize}
\item[$\edgegraph$]
If $\left\langle \gamma \right\rangle$ is contained in a subgroup of type $\Kleinfourgroup$
inside~$\Gamma$, then any chain of edges for~$\gamma$ admits endpoints of stabilizer types~$\Kleinfourgroup$ or~$\Afour$,
because we can merge any two adjacent edges on a $2$--torsion axis with touching point of stabilizer type $\Z/2$ or $\Sthree$.
As~$\Kleinfourgroup$ is an Abelian group and the reflections in~$\Afour$ are contained in the normal subgroup~$\Kleinfourgroup$,
the reflections in these endpoint stabilizers commute with~$\gamma$,
so the quotient space $\Hy^\gamma /_{C_\Gamma(\gamma)}$ is represented by a chain of edges for~$\gamma$.
What remains to show, is that there is no element of~${C_\Gamma(\gamma)}$ identifying the two endpoints of stabilizer type
$\Kleinfourgroup$ (respectively~$\Afour$).
Assume that there is an element $g \in {C_\Gamma(\gamma)}$ carrying out this identification.
Any one of the two endpoints, denote it by $x$, contains in its stabilizer a reflection $\alpha$ of the rotation axis of $\gamma$.
The other endpoint is then $g \cdot x$ and contains in its stabilizer the conjugate $^g \alpha$ by $g$.
Denote by $m$ the point in the middle of $(x,g\cdot x)$, i.e. the point on $\Hy^\gamma$ with equal distance to $x$ and to $g \cdot x$.
As $\left\langle ^g \alpha, \gamma \right\rangle$ is Abelian, $^g \alpha$ is in~${C_\Gamma(\gamma)}$ and hence $(x,m)$ and $(g \cdot x, m)$ 
are equivalent modulo~${C_\Gamma(\gamma)}$ via the element  $^g \alpha g $ 
(illustration: \fouredges). 
Then the chain of edges for $\gamma$ does not reach from $x$ to $g \cdot x$.
This contradicts our hypotheses, so the homeomorphism type of $\Hy^\gamma /_{C_\Gamma(\gamma)}$  is an edge without identifications.
\item[$\circlegraph$]
The other case is analogous to the first case of the proof of Theorem~\ref{3-torsion quotients},
 the r\^{o}le of $\Sthree$ being  played by $\Kleinfourgroup$ and $\Afour$.
\end{itemize}
\end{proof}

Furthermore, the following easy-to-check statement will be useful for our orbifold cohomology computations.
\begin{remark} \label{number of conjugacy classes in finite subgroups}
There is only one conjugacy class of elements of order $2$ in $\Sthree$ as well as in $\Afour$. In $\Sthree$,
 there is also only one conjugacy class of elements of order $3$,
 whilst in $\Afour$ there is an element $\gamma$ such that $\gamma$ and $\gamma^2$
 represent the two conjugacy classes of elements of order $3$.
\end{remark}
\begin{proof}
In cycle type notation, we can explicitly establish the multiplication tables of $\Sthree$ and $\Afour$,
 and compute the conjugacy classes.
\end{proof}

\begin{corollary}[Corollary to Remark~\ref{number of conjugacy classes in finite subgroups}]
\label{number of elements modulo conjugacy}
Let $\gamma$ be an element of order $3$ in a Bianchi group~$\Gamma$ with units $\{\pm 1\}$.
Then, $\gamma$ is conjugate in~$\Gamma$ to its square $\gamma^2$ if and only if there exists a group $G \cong \Sthree$ with $\left\langle \gamma \right\rangle \subsetneq G \subsetneq \Gamma$. 
\end{corollary}

Denote by $\lambda_{2\ell}$ the number of conjugacy classes of subgroups of type $\Z/_{\ell\Z}$
in a finite index subgroup {$\Gamma$} in a Bianchi group with units $\{\pm 1 \}$.
Denote by $\lambda_{2\ell}^*$ the number of conjugacy classes of those of them which are contained in a subgroup of type $\Dtwon$
in~$\Gamma$.
By Corollary~\ref{number of elements modulo conjugacy},
there are \mbox{$2\lambda_6 -\lambda_6^*$} conjugacy classes of elements of order $3$.
As a result of Theorems~\ref{3-torsion quotients} and~\ref{2-torsion quotients},
we have the following isomorphism of vector spaces:
\begin{eqnarray*}
&& \bigoplus_{\gamma \in T} \Homol^{\bullet} 
\left( (\Hy_\R)^{\gamma}/C_\Gamma (\gamma) ; \mathbb{Q} \right) \\
&& \cong  \Homol^{\bullet} \left( \Hy_\R / _\Gamma; \thinspace \rationals \right) 
\bigoplus\nolimits^{\lambda_4^*} \Homol^{\bullet} \left( \edgegraph; \thinspace \rationals \right)
\bigoplus\nolimits^{(\lambda_4 -\lambda_4^*)} \Homol^{\bullet} \left( \circlegraph; \thinspace \rationals \right)
\bigoplus\nolimits^{(2\lambda_6 -\lambda_6^*)} \Homol^{\bullet} \left(\circlegraph; \thinspace \rationals \right) \, ,
\end{eqnarray*}
where $T\subset \Gamma$ is a set of representatives of conjugacy classes 
of elements of finite order in $\Gamma$.
The (co)homology of the quotient space $\Hy_\R  / _\Gamma$
 has been computed numerically for a large scope of Bianchi groups \cite{Vogtmann}, \cite{Scheutzow}, \cite{Higher_torsion};
 and bounds for its Betti numbers have been given in \cite{KraemerThesis}.
Kr\"amer \cite{Kraemer} has determined number-theoretic formulae
 for the numbers $\lambda_{2\ell}$ and $\lambda_{2\ell}^*$ of conjugacy classes of finite subgroups in the full Bianchi groups.
Kr\"amer's formulae have been evaluated for hundreds of thousands of Bianchi groups \cite{accessing_Farrell_cohomology}, 
and these values are matching with the ones from the orbifold structure computations with \cite{BianchiGP}
in the cases where the latter are available.

When we pass to the complexified orbifold $[\Hy^3_\C / \Gamma]$,
the real line that is the rotation axis in~$\Hy_\R$ of an element of finite order, becomes a complex line. 
However, the centralizer still acts in the same way by reflections and translations.
So, the interval $\edgegraph$ as a quotient of the real line yields a stripe
 $\edgegraph \times \R$ as a quotient of the complex line. 
And the circle $\circlegraph$ as a quotient of the real line yields a cylinder
 $\circlegraph \times \R$ as a quotient of the complex line.
Therefore,  using the degree shifting numbers of Lemma~\ref{rotationAge}, we obtain the result of Corollary~\ref{introduced result},
$$ \Homol^d_{\rm CR}\left([\Hy^3_\C / \Gamma] \right) \cong 
\Homol^d\left(\Hy_\C / _{\Gamma}; \thinspace \rationals \right) \oplus
\begin{cases}
\rationals^{\lambda_4 +2\lambda_6 -\lambda_6^*}, & d=2, \\
 \rationals^{\lambda_4-\lambda_4^* +2\lambda_6 -\lambda_6^*}, & d=3, \\
0, & \mathrm{otherwise}. \end{cases}$$ 

\bigskip

As the authors have calculated the Bredon homology $\Homol^\mathfrak{Fin}_0(\Gamma; R_\C)$ of the Bianchi groups
with coefficients in the complex representation ring functor $R_\C$ (see \cite{RahmAIF}), 
Mislin's following lemma allows us a verification of our computations (we calculate both sides of Mislin's isomorphism explicitly).

\begin{lemma}[Mislin \cite{MislinValette}] \label{numberOfConjugacyClasses}
 Let $\Gamma$ be an arbitrary group and write ${\rm FC}(\Gamma)$ for the set of conjugacy classes of elements of finite order in $\Gamma$.
Then there is an isomorphism
$$\Homol^\mathfrak{Fin}_0(\Gamma; R_\C) \otimes_\Z \C \cong \C[{\rm FC}(\Gamma)].$$
\end{lemma}

\section{The cohomology ring isomorphism} \label{ringIsomorphism}

In this section, we prove Theorem~\ref{mainthm}. To this aim, we first prove that there is 
a bijective correspondence between conjugacy classes of elements of finite order in $\Gamma \setminus \{ 1 \}$
and  exceptional prime divisors of the crepant resolution $f\colon Y \to  \Hy^3_\C / \Gamma$.
Here we follow, and we use results from, \cite{IR}, therefore we interpret the aforementioned 
correspondence as a McKay correspondence for complexified Bianchi orbifolds.
In Section~\ref{Phi}, we use this correspondence to define a morphism of graded  vector spaces
$\Phi \colon \Homol^*_{\rm CR}([\Hy^3_\C / \Gamma]) \to \Homol^* (Y)$.
Finally, using a Mayer-Vietoris argument, together with results from \cite{IR} and \cite{Perroni}, we show that $\Phi$ is an isomorphism
and  that it preserves the cup products. 

Throughout this section, $\Gamma$ is a Bianchi group and $[\Hy^3_\C /\Gamma]$
is the corresponding complexified Bianchi orbifold.

\subsection{The singular locus of complexified Bianchi orbifolds and the existence of crepant resolutions}
Let us recall that the singular points of $\Hy^3_\C / \Gamma$ are the image, under the projection
$\Hy^3_\C  \to \Hy^3_\C / \Gamma$, of the points with non-trivial stabilizer.
Moreover, every element  $\gamma \in \Gamma \setminus \{ 1\}$, such that $(\Hy^3_\C)^\gamma
\not= \emptyset$, is a non-trivial rotation of $\Hy^3_\C$ of order 
$2$ or $3$ (see the discussion before Lemma~\ref{rotationAge})
and the fixed-point locus $(\Hy^3_\C)^\gamma$ is a Riemann surface. 
More precisely, we get the following result. 
\begin{lemma}\label{Sigma gamma}
Let $\Sigma \subset \Hy^3_\C / \Gamma$ be the singular locus. Then, the following 
statements hold true.
\begin{itemize}
\item[1)] 
$\Sigma$ is an analytic space of dimension $1$ with finitely many singular points
$x_1, \ldots , x_s$.
\item[2)] 
For any $\gamma \in \Gamma$, let 
$\imath_\gamma \colon (\Hy^3_\C)^\gamma / C_\Gamma (\gamma) \to \Hy^3_\C/\Gamma$
be the morphism induced by the inclusion $(\Hy^3_\C)^\gamma \hookrightarrow \Hy^3_\C$
and let $\Sigma_\gamma := {\rm Im}(\imath_\gamma)$
be the image of $\imath_\gamma$. 
Then, every irreducible component of $\Sigma$ is equal to  
$\Sigma_\gamma \subset \Sigma$,
for some $\gamma \in \Gamma$.  
\item[3)]
For any $\gamma \in \Gamma$, the centralizer $C_\Gamma (\gamma)$ is a normal subgroup 
of $N_\Gamma (\left\langle \gamma \right\rangle)$, 
the normalizer of $\left\langle \gamma \right\rangle$ in~$\Gamma$.
Moreover, $N_\Gamma (\left\langle \gamma \right\rangle) / C_\Gamma (\gamma)$
acts on $(\Hy^3_\C)^\gamma / C_\Gamma (\gamma)$. And
$(\Hy^3_\C)^\gamma / N_\Gamma (\left\langle \gamma \right\rangle)$ is the normalisation 
of $\Sigma_\gamma$.
\item[4)]
Let $\gamma \in \Gamma$. If $\gamma$ has order $2$, or it has order $3$
and it is not conjugated to $\gamma^2$ in $\Gamma$, 
then $C_\Gamma (\gamma) = N_\Gamma (\left\langle \gamma \right\rangle)$.
If $\gamma$ has order $3$ and it is conjugated to $\gamma^2$  in $\Gamma$,
then $C_\Gamma (\gamma)$ has index $2$ in $N_\Gamma (\left\langle \gamma \right\rangle)$.
\end{itemize}  
\end{lemma}
\begin{proof}
1) As observed before, if 
$\gamma \in \Gamma \setminus \{ 1\}$ is such that $(\Hy^3_\C)^\gamma \not= \emptyset$,
then it is a non-trivial rotation. Therefore, $(\Hy^3_\C)^\gamma$ is a Riemann surface
and so $\Sigma$ is an analytic space of dimension $1$. 
The singular points of $\Sigma$ are the image of the points 
$z\in \Hy^3_\C$ with stabilizer not cyclic. As observed in Remark 
\ref{rem spine 2}, such points belong to $\Hy^3_\R$. 
Now, the fact that $\Sigma$  has finitely many singular points follows from the existence 
of a fundamental domain for the action of $\Gamma$ on $\Hy^3_\R$, which is bounded
by finitely many geodesic surfaces (\cite[Thm. 1.1, p. 311]{ElstrodtGrunewaldMennicke}).

2) Is a consequence of the fact that $(\Hy^3_\C)^\gamma$ is  irreducible,
as it is isomorphic to the disk $\Delta = \{ z\in \C \, | \, | z | <1\}$, and the 
image of an irreducible analytic space is irreducible.

3) Let $\eta \in C_\Gamma (\gamma)$, and let $\delta \in N_\Gamma (\left\langle
\gamma \right\rangle)$. Then 
$$
\delta^{-1}\eta \delta \gamma = \delta^{-1}\eta \gamma^k \delta = \delta^{-1}
\gamma^k \eta \delta = \gamma \delta^{-1}
 \eta \delta \, ,
$$
where $k\in \NN$ is such that $\delta \gamma \delta^{-1}= \gamma^k$. 
Hence $\delta^{-1}\eta \delta \in C_\Gamma (\gamma)$ and so 
$C_\Gamma (\gamma)$ is a normal subgroup of $N_\Gamma (\left\langle
\gamma \right\rangle)$. \\
The natural action of $N_\Gamma (\left\langle \gamma \right\rangle)$ 
on $(\Hy^3_\C)^\gamma$ is properly discontinuous, hence every point has finite stabilizer.
From this it follows that  $(\Hy^3_\C)^\gamma / N_\Gamma (\left\langle \gamma \right\rangle)$ is a normal analytic space.
Furthermore, let $z, z' \in (\Hy^3_\C)^\gamma$ be two points that are 
mapped to the same point  $x\in \Hy^3_\C / \Gamma$,
and  suppose  that $x$ is a smooth point of $\Sigma$.
Under these hypotheses, ${\rm Stab}_\Gamma ( z ) = {\rm Stab}_\Gamma ( z' ) 
=\left\langle \gamma \right\rangle$ and so, if $g\in \Gamma$ is such that $g\cdot z = z'$,
we know that $g\left\langle \gamma \right\rangle g^{-1} = \left\langle \gamma \right\rangle$, that is $g\in N_\Gamma (\left\langle \gamma \right\rangle)$.  
This implies that $\imath_\gamma$ induces a birational map 
between $(\Hy^3_\C)^\gamma / N_\Gamma (\left\langle \gamma \right\rangle)$
and $\Sigma_\gamma$, hence 3) follows.

To prove 4), let us consider the  action of  
$N_\Gamma(\left\langle \gamma \right\rangle)/C_\Gamma (\gamma)$
on $\left\langle \gamma \right\rangle \setminus \{ 1 \}$ given by conjugation. 
If $\gamma$ has order $2$, or it has order $3$ and it is not conjugated to $\gamma^2$,
then this action is trivial, hence $C_\Gamma (\gamma) = N_\Gamma(\left\langle \gamma \right\rangle)$.
In the remaining case, the orbit of $\gamma$ has two elements, so the result follows. 
\end{proof}

The existence of a crepant resolution of $\Hy^3_\C/\Gamma$ follows from \cite{Roan} (see also \cite{ChenTseng}),
since $\Hy^3_\C/\Gamma$ has Gorenstein singularities (Lemma~\ref{stab in SL3}). 
For later use, and to fix notations, we briefly review its construction. 
Under the notation of Lemma~\ref{Sigma gamma}, let $x_1, \ldots , x_s \in \Sigma$ be the singular points 
of $\Sigma$ (the singular locus of $\Hy^3_\C/\Gamma$).
By Lemma~\ref{stab in SL3}, there are  disjoint open neighborhoods
$U_1, \ldots, U_s \subset \Hy^3_\C/\Gamma$
of $x_1, \ldots , x_s$, respectively, each of them isomorphic to the quotient
of an open neighborhood of the origin in $\C^3$ by a finite subgroup of
$SL_3 (\C)$. Therefore, for any $i=1, \ldots , s$, there exists
a crepant resolution $f_i \colon V_i \to U_i$ of  $U_i$. 

Let $X:= (\Hy^3_\C/\Gamma ) \setminus \{ x_1, \ldots , x_s \}$ be the complement 
of $x_1, \ldots , x_s$. It is an analytic space with {\bf transverse singularities of type $A$}.
That is, every singular point $x\in X$ has a neighborhood isomorphic to
a neighborhood of a singular point of $\{ (u,v,w) \in \C^3 \, | \, w^{n+1}= uv\}
\times \C^{d-2}$, for some integer $n\geq 1$, where $d=\dim ( X )$ 
is equal to $3$ in our case. 
Notice that $n$ is constant on each connected component $C$ of the singular 
locus of $X$, hence we say that $X$ has transverse singularities of type $A_n$
on $C$. 

Every analytic space with transverse singularities of type $A$
admits a unique crepant resolution (see e.g. \cite[Prop. 4.2]{Perroni}),
up to canonical isomorphism.
So, let $f_0 \colon V \to X$ be a crepant resolution of $X$. 
By  uniqueness,  the restriction 
of $f_0 \colon V \to X$ to $U_i \setminus \{ x_i\}$ is canonically isomorphic to 
the restriction of $f_i \colon V_i \to U_i$ to  $U_i \setminus \{ x_i\}$, $\forall i=1, \ldots , s$.
Therefore, $f_0, f_1, \ldots , f_s$ can be glued together, yielding a crepant resolution 
$f\colon Y \to \Hy^3_\C/\Gamma$. 

\subsection{McKay correspondence for complexified Bianchi orbifolds}
In this section, we prove that there is a natural one-to-one correspondence 
between conjugacy classes of elements of finite order of $\Gamma \setminus \{ 1 \}$
and exceptional prime divisors of the crepant resolution $f\colon Y \to \Hy^3_\C/\Gamma$.
Let us recall that the authors of~\cite{IR} 
define a natural bijective correspondence between conjugacy classes of junior elements 
of $G\setminus \{ 1 \}$ (here, $G$ is a finite subgroup of $SL_n(\C)$) and exceptional prime divisors 
of a minimal model of $\C^n /G$ (if $f\colon Y \to \C^n /G$ is a crepant resolution, then
$Y$ is a minimal model of $\C^n /G$). This result has been interpreted, and extended in several directions, 
using derived categories in \cite{BridgelandKingReid} and \cite{ChenTseng}.

We will need some general facts about analytic spaces with transverse 
singularities of type $A$, which we briefly recall.
Let $M$ be a complex manifold with an effective and properly discontinuous action 
of a discrete group $\Gamma$, and let $X=M/\Gamma$
be the quotient space. For $\gamma \in \Gamma$, let $C\subset X$ 
be the image under the quotient map 
$M\to M/\Gamma$ of the fixed-point locus $M^\gamma$.
Let us suppose that $X$ has transverse singularities  of type $A_n$ on $C$.
In particular, the stabilizer of any point 
$z\in M^\gamma$ is $\left\langle \gamma \right\rangle \cong \ZZ/(n+1)\ZZ$,
so two points $z, z' \in M^\gamma$ are identified by the projection  $M\to M/\Gamma$
if and only if they are on the same $N_\Gamma (\left\langle \gamma \right\rangle)$-orbit
(see the proof of Lemma~\ref{Sigma gamma}, 3), where 
$N_\Gamma (\left\langle \gamma \right\rangle)$ is the normalizer 
of $\left\langle \gamma \right\rangle$ in $\Gamma$. 
\begin{note}
We observe that, in this situation, for any $g\in N_\Gamma (\left\langle \gamma \right\rangle)$,
$g\gamma g^{-1} = \gamma^{\pm 1}$. 
\end{note}
\begin{proof}
Let us consider the normal bundle of $M^\gamma$ in $M$, $N_{M^\gamma / M}$.
The group $\left\langle \gamma \right\rangle$ acts fibrewise on $N_{M^\gamma / M}$,
so we have a splitting $N_{M^\gamma / M} = (N_{M^\gamma / M})^{\chi} \oplus 
(N_{M^\gamma / M})^{\chi^{-1}}$, where $(N_{M^\gamma / M})^{\chi^{\pm 1}}$
is the sub-bundle  of $N_{M^\gamma / M}$ 
where $\left\langle \gamma \right\rangle$ acts as multiplication by the character 
$\chi^{\pm 1}$, and $\chi$ is a generator of the group of characters of 
$\left\langle \gamma \right\rangle$. Assume that $g\gamma g^{-1} = \gamma^k$, 
and let $z\in M^\gamma$, $z':=g\cdot z$. Then, the tangent map
of $g$ at $z$, $T_z g$, induces an isomorphism 
\begin{equation}\label{Tg intertwines}
N_{M^\gamma / M}(z) \cong N_{M^\gamma / M}(z')
\end{equation}
between the fibre of $N_{M^\gamma / M}$ at $z$
and that at $z'$. Since $T_z g \circ T_z \gamma = T_z(g \circ \gamma ) =
T_{z'} \gamma^k \circ T_z g$, the  isomorphism \eqref{Tg intertwines}
yields an isomorphism between the following representations of $\left\langle \gamma \right\rangle$: 
$\left\langle \gamma \right\rangle \to GL(N_{M^\gamma / M}(z))$, $\gamma \mapsto
T_z \gamma$, and   $\left\langle \gamma \right\rangle \to GL(N_{M^\gamma / M}(z'))$, $\gamma \mapsto
T_{z'} \gamma^k$. But the last representation is the direct sum 
of the  irreducible representations of $\left\langle \gamma \right\rangle$ having 
characters $\chi^k$ and $\chi^{-k}$, so $k \equiv \pm 1 \, ({\rm mod} \, n+1)$. 
\end{proof}
We say that $X$ has {\bf transverse singularities of type $A_n$ on $C$
and non-trivial monodromy}, if $\gamma$ is conjugated to $\gamma^{-1}$
in $\Gamma$. Otherwise we say that $X$ has {\bf transverse singularities of type $A_n$ on 
$C$
and trivial monodromy}. We refer to \cite[Section 3.1]{Perroni} for an equivalent definition
of the monodromy. 
Notice also that in \cite{Perroni} the monodromy is referred to a suitable neighbourhood 
of $[M^\gamma / N_{\Gamma}(\left\langle \gamma \right\rangle)]$ in the  orbifold 
$[M/\Gamma]$. However, by \cite[Prop. 2.9]{Perroni}, such an orbifold structure
is determined uniquely by $X$.


Let now   $\tilde{U}$ be a  neighborhood
of $M^\gamma$ in $M$ that is isomorphic to a neighborhood of the 
$0$-section of $N_{M^\gamma / M}$ (i.e. a tubular neighborhood 
of $M^\gamma$ in $M$). The natural action of  $N_\Gamma (\left\langle \gamma \right\rangle)$
on $N_{M^\gamma / M}$ induces an action of 
$N_\Gamma (\left\langle \gamma \right\rangle)$ on 
$\tilde{U}$, such that $\tilde{U}/ N_\Gamma (\left\langle \gamma \right\rangle)$
is an open neighborhood of $C$ in $X$.
Moreover, if $X$ has non-trivial monodromy on $C$, then 
$\tilde{U}/C_\Gamma(\gamma)$ is an analytic space with transverse
singularities of type $A_n$  on $M^\gamma/C_\Gamma(\gamma)$
and trivial monodromy,
and the natural map $\tilde{U}/C_\Gamma(\gamma) \to
\tilde{U}/N_\Gamma (\left\langle \gamma \right\rangle)$ is a two-to-one topological covering
(this is analogous to \cite[Cor. 3.6]{Perroni}).

We summarise in the following proposition 
the previous considerations, in the case of complexified Bianchi orbifolds.
\begin{proposition}\label{Bianchi orbifolds monodromy}
Let $\Gamma$ be a Bianchi group, let $\Sigma$
be the singular locus of $\Hy^3_\C / \Gamma$, and let $x_1, \ldots , x_s$
be the singular points of $\Sigma$. Let $X:= (\Hy^3_\C / \Gamma) \setminus \{ 
x_1, \ldots , x_s\}$. Then, the following holds true.
\begin{itemize}
\item[1)]
For every connected component $C$ of $\Sigma \setminus \{ x_1, \ldots , x_s\}$,
$X$ has transverse singularities of type $A_n$ on $C$, with $n \in \{1,2\}$.
\item[2)]
If $C$ is contained in the image of $(\Hy^3_\C)^\gamma$ with
$n=2$ and $\gamma$ is not conjugated to $\gamma^2$, or with $n=1$, 
then $X$ has trivial monodromy on $C$.
\item[3)]
If $n=2$ and $\gamma$ is conjugated to $\gamma^2$, then $X$ has non-trivial
monodromy on $C$. Furthermore, using the same notation as  before, 
$\tilde{U}/C_\Gamma (\gamma)$ is an analytic space with transverse singularities 
of type $A_2$ and trivial monodromy; the map $\tilde{U}/C_\Gamma (\gamma)
\to \tilde{U}/N_\Gamma (\left\langle \gamma\right\rangle)$ is a two-to-one topological covering.      
\end{itemize}
\end{proposition}

In the following proposition, we establish a McKay correspondence 
for complexified Bianchi orbifolds following \cite{IR}.
\begin{proposition}\label{McKayforBianchi}
Let $\Gamma$ be a Bianchi group, and let $f\colon Y \to \Hy^3_\C / \Gamma$
be a crepant resolution. Then, there is a one-to-one correspondence 
between conjugacy classes of elements of finite order of $\Gamma \setminus \{ 1 \}$
and exceptional prime divisors of $f$.
\end{proposition}
\begin{proof}
Let $\gamma \in \Gamma \setminus \{ 1 \}$ be an element of finite 
order. Then $\gamma$ is an elliptic element of
${\rm PSL}_2(\C)$ (\cite[Def. 1.3, p. 34]{ElstrodtGrunewaldMennicke})
and so  $(\Hy^3_\C)^\gamma \not= \emptyset$ 
(\cite[Prop. 1.4, p. 34]{ElstrodtGrunewaldMennicke}).
By Lemma~\ref{rotationAge}, the degree shifting number of $\gamma$ is $1$,
in other words, in the notation of \cite{IR}, 
$\gamma$ is a junior element of ${\rm Stab}_\Gamma ( z )$,
for any $z\in (\Hy^3_\C)^\gamma$.

As observed in Lemma~\ref{Sigma gamma},
the image of $(\Hy^3_\C)^\gamma$ in $\Hy^3_\C / \Gamma$
is an irreducible component $\Sigma_\gamma$ of the singular 
locus $\Sigma$ of $\Hy^3_\C / \Gamma$.

Now we distinguish two cases.\\
{\bf Case 1:} $\gamma$ has order $2$, or it has order $3$ and is not conjugated 
to $\gamma^2$ in $\Gamma$. \\
Then, there are open subsets $U_1, \ldots , U_r \subset \Hy^3_\C / \Gamma$
of the form  $U_i = \tilde{U}_i / G_i$, where $\tilde{U}_i \subset \Hy^3_\C$
is an open neighbourhood of a point $z\in (\Hy^3_\C)^\gamma$,
$G_i = {\rm Stab}_\Gamma ( z )$, for $i=1, \ldots , r$,
and such that $\Sigma_\gamma \subset \cup_{i=1}^r U_i$.
Notice that in this case, the conjugacy class of $\gamma \in G_i$
consists only of $\gamma$, so by \cite[Cor. 1.5]{IR},  $\gamma$   
corresponds to an exceptional prime divisor $E_{\gamma, i}$
of the restriction $f_{|f^{-1}(U_i)} \colon f^{-1}(U_i) \to U_i$,
for any $i=1, \ldots , r$ (as observed before, $\gamma$ is a junior element 
of $G_i$). Moreover, by the definition of 
the $E_{\gamma, i}$'s (see \cite{IR}), it follows that on  
$f^{-1}(U_i \cap U_j)$, the divisors $E_{\gamma, i}$ and $E_{\gamma, j}$
coincide, so they glue together to form an exceptional prime 
divisor $E_\gamma \subset Y$ of $f$.

{\bf Case 2:} $\gamma$ has order  $3$ and is  conjugated 
to $\gamma^2$ in $\Gamma$. \\
Let now $C \subset \Sigma_\gamma$ be the complement 
in $\Sigma_\gamma$ of the singular points of $\Sigma$.
By Proposition~\ref{Bianchi orbifolds monodromy}, there is an open subset 
$\tilde{U} \subset \Hy^3_\C$, with an action of $N_\Gamma (\left\langle \gamma \right\rangle)$,
such that $\tilde{U} / N_\Gamma (\left\langle \gamma \right\rangle) \subset 
\Hy^3_\C / \Gamma$ is an open neighbourhood of $C$
and 
$\tilde{U} / C_\Gamma (\gamma) \to 
\tilde{U} / N_\Gamma (\left\langle \gamma \right\rangle)$ is a two-to-one topological covering.
Furthermore, $\tilde{U} / C_\Gamma (\gamma)$ is an analytic space with transverse
singularities of type $A_2$ and trivial monodromy. 
Let $\tilde{V} \to \tilde{U} / C_\Gamma (\gamma)$ be a crepant resolution,
then by the uniqueness of the crepant resolution for spaces with transverse 
singularities of type $A$, there is a   
morphism 
\begin{equation}\label{vtildetoy}
\tilde{V} \to f^{-1} \left(\tilde{U} / N_\Gamma (\left\langle \gamma \right\rangle)\right) \, ,
\end{equation}
such that the following diagram commutes:
$$
\begin{CD}
\tilde{V} @>>> f^{-1} \left((\tilde{U} / N_\Gamma (\left\langle \gamma \right\rangle)\right) \\
@VVV @VV{f_|}V \\
\tilde{U} / C_\Gamma (\gamma) @>>> \tilde{U} / N_\Gamma (\left\langle \gamma \right\rangle) \, .
\end{CD}
$$
From Case 1, $\gamma$ corresponds to an exceptional prime 
divisor $F_\gamma$ of $\tilde{V} \to \tilde{U} / C_\Gamma (\gamma)$.
Let us denote by $E_{\gamma,0} \subset Y$ the image of $F_\gamma$
under the morphism \eqref{vtildetoy}.

In order to extend $E_{\gamma, 0}$ over the whole $\Sigma_\gamma$,
let $W\subset \Hy^3_\C/\Gamma$ be a (possibly disconnected) neighborhood
of $\Sigma_\gamma \setminus C$, such that each connected component 
is of the form $\tilde{W}/G$, where $\tilde{W} \subset \Hy^3_\C$
is an open neighborhood of a point $z\in (\Hy^3_\C)^\gamma$, and 
$G={\rm Stab}_\Gamma ( z )$. By \cite[Cor. 1.5]{IR}, 
(the conjugacy class of) $\gamma$ corresponds to an exceptional  prime divisor
$E'_\gamma \subset f^{-1}(W)$. By construction, $E_{\gamma, 0}$
and $E'_\gamma$ glue together to form an exceptional prime divisor 
$E_\gamma$ of $f\colon Y \to \Hy^3_\C/\Gamma$.  \\
Notice that if we apply the same procedure starting from 
$\gamma^2=\gamma^{-1}$,  we obtain the same divisor $E_\gamma$.
This concludes the proof of the proposition. 
\end{proof}

\subsection{The linear map}\label{Phi}
Let $\Gamma$ be a Bianchi group, and let $f\colon Y \to \Hy^3_\C / \Gamma$ 
be a crepant resolution of $\Hy^3_\C /\Gamma$.
In this section, we define a linear map 
\begin{eqnarray*}
\Phi \colon \Homol^*_{\rm CR}([\Hy^3_\C / \Gamma], \QQ) \to \Homol^* (Y, \QQ) \, .
\end{eqnarray*}
To this aim, let us fix the following presentation of the Chen-Ruan orbifold cohomology of 
$[\Hy^3_\C / \Gamma]$ (cf. Definition \ref{HCR}):
\begin{equation}\label{HCRbis}
 \Homol^*_{\rm CR}([\Hy^3_\C / \Gamma], \QQ) = 
\oplus_{\gamma \in T} \Homol^{*-2{\rm shift}(\gamma)} 
\left( (\Hy^3_{\C})^\gamma / C_\Gamma (\gamma) , \QQ \right)
\end{equation}
where $T\subset \Gamma$ is a set of representatives of the 
conjugacy classes of elements of finite order of $\Gamma$.
Then $\Phi$ is defined as the sum of linear maps 
\begin{equation}\label{HCRter}
\Phi_\gamma \colon \Homol^{*-2{\rm shift}(\gamma)} 
\left( (\Hy^3_{\C})^\gamma / C_\Gamma (\gamma) , \QQ \right)
 \to \Homol^* (Y, \QQ) \, , \quad {\rm for} \quad  \gamma \in T \, .
\end{equation}

If $\gamma = 1$, then we define $\Phi_1 := f^* \colon 
\Homol^{*}
\left( \Hy^3_{\C}/ \Gamma  , \QQ \right)
 \to \Homol^* (Y, \QQ)$. Let now $\gamma \in T\setminus \{ 1 \}$,
 and consider the following Cartesian diagram 
$$
\begin{CD}
\tilde{E}_\gamma @>{\tilde{\jmath}_\gamma}>> E_\gamma \\
@V{\pi}VV @VV{f_{|E_\gamma}}V \\
(\Hy^3_{\C})^\gamma / C_\Gamma (\gamma) @>{\imath_\gamma}>> 
\Hy^3_{\C}/ \Gamma \, , 
\end{CD}
$$
where $\imath_\gamma$ 
is the morphism induced by the inclusion $(\Hy^3_{\C})^\gamma \hookrightarrow
\Hy^3_{\C}$, and $E_\gamma \subset Y$ is the exceptional prime divisor 
corresponding to the class of $\gamma$ by Proposition~\ref{McKayforBianchi}.
Let $\jmath_\gamma \colon \tilde{E}_\gamma \to Y$ be the composition of 
$\tilde{\jmath}_\gamma  \colon \tilde{E}_\gamma \to E_\gamma$ 
followed by the inclusion $E_\gamma \hookrightarrow Y$. Notice that $\jmath_\gamma$ is proper
since $\imath_\gamma$ is so (Lemma~\ref{ProperTwistedSectors}). Then we define
\begin{equation}\label{HCRter'}
\Phi_\gamma (\alpha):= (\jmath_\gamma )_* (\pi^* (\alpha)) \, , \quad 
\forall \alpha \in \Homol^{*-2{\rm shift}(\gamma)} 
\left( (\Hy^3_{\C})^\gamma / C_\Gamma (\gamma) , \QQ \right) \, .
\end{equation}
Where, $\forall \beta \in \Homol^*(\tilde{E}_\gamma, \QQ)$,
$(\jmath_\gamma)_*(\beta) \in \Homol^{*+2}(Y, \QQ)$ is the cohomology class
that corresponds via Poincar\'e duality (cf. \cite[Chapter XIV]{Massey}) to the following element of
$\Homol_c^{4-*}(Y, \QQ)^\vee$ (the dual space of the cohomology 
of $Y$ with compact support):
\begin{equation}\label{pushforward}
\omega \in \Homol_c^{4-*}(Y, \QQ) \mapsto 
\int_{\tilde{E}_\gamma} \beta \cup \jmath_\gamma^* (\omega) \, .
\end{equation}
\begin{remark}\normalfont
In \eqref{pushforward}, $\tilde{E}_\gamma$ is a complex analytic space 
of real dimension $4$ (it is a divisor of $Y$).
If it is singular,  by the integral  
$\int_{\tilde{E}_\gamma} \beta \cup \jmath_\gamma^* (\omega)$
we mean the integral of the pull-back of $\beta \cup \jmath_\gamma^* (\omega)$
on a resolution of the singularities of $\tilde{E}_\gamma$ 
(which is a complex manifold and hence it has a natural orientation). 
Notice that this does not depend on the particular resolution of 
$\tilde{E}_\gamma$. If 
$\rho' \colon \tilde{E}_\gamma' \to \tilde{E}_\gamma$ and 
$\rho'' \colon \tilde{E}_\gamma'' \to \tilde{E}_\gamma$ are two  resolutions 
of $\tilde{E}_\gamma$, then there exists a third resolution 
$\tilde{E}_\gamma'''$, with two morphisms 
$\rho_1 \colon \tilde{E}_\gamma''' \to \tilde{E}_\gamma'$,
$\rho_2 \colon \tilde{E}_\gamma''' \to \tilde{E}_\gamma''$, such that 
$\rho' \circ \rho_1 = \rho'' \circ \rho_2$. 
One can take, for example, $\tilde{E}_\gamma'''$ to be a resolution of the Cartesian product 
$\tilde{E}_\gamma' \underset{\rho' , \tilde{E}_\gamma, \rho''} \times \tilde{E}_\gamma''$.
In particular, $\tilde{E}_\gamma'''$ differs from $\tilde{E}_\gamma'$
($\tilde{E}_\gamma''$, respectively) by a closed analytic subspace of (complex) codimension $\geq 1$,
which has measure zero and so the integral in \eqref{pushforward}
does not depend on the resolution of  $\tilde{E}_\gamma$.
\end{remark}

Let us first notice that $\Phi$ is degree preserving, since any $\gamma \in \Gamma \setminus \{ 1 \}$
has ${\rm shift}(\gamma) =1$, and $(\jmath_\gamma)_* \colon \Homol^*(\tilde{E}_\gamma, \QQ) \to \Homol^{*+2}(Y, \QQ)$
increases the degrees by two (the real codimension of $\tilde{E}_\gamma$ in $Y$). 

In the proof of Theorem~\ref{mainthm} we will use a compatibility property of $\Phi$
with respect to open embeddings, as follows.
Let $U\subset \Hy^3_\C /\Gamma$ be an open subset, and let
$\tilde{U}\subset \Hy^3_\C$ be the pre-image of $U$ with respect to the quotient map $\Hy^3_\C \to \Hy^3_\C /\Gamma$.
Then the action of $\Gamma$ on $\Hy_\C^3$ restricts to an action on $\tilde{U}$, in such  a way that 
$[\tilde{U}/\Gamma]$ is an open sub-orbifold of $[\Hy^3_\C/\Gamma]$. 
The same definition of $\Phi$ gives a linear map
$$
\Phi^U \colon \Homol^*_{\rm CR}([\tilde{U} / \Gamma], \QQ) \to \Homol^* (f^{-1}(U), \QQ) \, .
$$  
\begin{lemma}\label{Phisheaf}
Under the previous notation, let $i\colon [\tilde{U}/\Gamma] \hookrightarrow [\Hy^3_\C /\Gamma]$ 
and  $j\colon f^{-1}(U) \hookrightarrow Y$ be the open inclusions.
Then
$$
\Phi^U \circ i^* = j^* \circ \Phi \, .
$$
\end{lemma}
\begin{proof}
It suffices to prove that $\Phi^U_\gamma \circ i^* = j^* \circ \Phi_\gamma$, for any $\gamma \in T$
of finite order, where $\Phi^U_\gamma$ and $\Phi_\gamma$ are defined as in \eqref{HCRter'}. 
If $\gamma =1$  the claim follows by the functoriality property of the pull-back. So, let us assume that 
$\gamma \not= 1$ and consider the following commutative diagram.
$$
\xymatrix{
\tilde{E}^U_\gamma \ar[d]_{\pi_{|}} \ar[r]^{\jmath^U_\gamma} \ar[rd]^{\tilde{i}} & f^{-1}(U) \ar[rd]^j &\\
\tilde{U}^\gamma / C_\Gamma (\gamma) \ar[rd]^i & \tilde{E}_\gamma \ar[d]^\pi \ar[r]^{\jmath_\gamma} & Y \\
& (\Hy^3_{\C})^\gamma / C_\Gamma (\gamma) & } 
$$
where, by abuse of notation, we have denoted with $i \colon \tilde{U}^\gamma / C_\Gamma (\gamma)  \to (\Hy^3_{\C})^\gamma / C_\Gamma (\gamma)$
the map induced by the inclusion $i\colon [\tilde{U}/\Gamma] \hookrightarrow [\Hy^3_\C /\Gamma]$;
$\pi$, $\tilde{E}_\gamma$ and $\jmath_\gamma$ are defined as in the definition of $\Phi$;
$\tilde{E}^U_\gamma := \pi^{-1}(\tilde{U}^\gamma / C_\Gamma (\gamma))$, $\pi_{|}$ is the restriction of $\pi$,
$\tilde{i}$ is the open inclusion, and $\jmath_\gamma^U$ is the restriction of $\jmath_\gamma$.
The result follows, if we prove that $(\jmath^U_\gamma)_* \circ \tilde{i}^* = j^* \circ (\jmath_\gamma)_*$.
So, let $\beta \in \Homol^* (\tilde{E}_\gamma)$. Then, $\forall \delta \in \Homol_c^*(f^{-1}(U))$, we get:
\begin{eqnarray*}
\int_{f^{-1}(U)} [(\jmath^U_\gamma)_* \circ \tilde{i}^*](\beta) \cup \delta &=& \int_{\tilde{E}^U_\gamma} \tilde{i}^*(\beta) \cup (\jmath^U_\gamma)^* (\delta) \, .
\end{eqnarray*}
On the other hand, there exists $\tilde{\delta} \in \Homol_c^*(Y)$ such that $\delta = j^*\tilde{\delta}$
(this follows from the excision property \cite[pp. 320, 362, 363]{Massey}).
Therefore, 
\begin{eqnarray*}
\int_{\tilde{E}^U_\gamma} \tilde{i}^*(\beta) \cup (\jmath^U_\gamma)^* (\delta) &=&
\int_{\tilde{E}^U_\gamma} \tilde{i}^*(\beta) \cup (\jmath^U_\gamma)^* (j^*\tilde{\delta}) \\
&=& \int_{\tilde{E}^U_\gamma} \tilde{i}^*(\beta) \cup
\tilde{i}^*(\jmath_\gamma^* (\tilde{\delta})) \\
&=& \int_{\tilde{E}^U_\gamma} \tilde{i}^*[\beta \cup \jmath_\gamma^* (\tilde{\delta})] \\
&=& \int_{\tilde{E}_\gamma} \beta \cup \jmath_\gamma^* (\tilde{\delta})\\
&=& \int_Y (\jmath_\gamma)_*(\beta) \cup \tilde{\delta} = \int_{f^{-1}(U)} [j^* \circ (\jmath_\gamma)_*](\beta) \cup \delta \, .
\end{eqnarray*} 
Using  Poincar\'e duality, we conclude that $[(\jmath^U_\gamma)_* \circ \tilde{i}^*](\beta) = [j^* \circ (\jmath_\gamma)_*](\beta)$.
\end{proof}

\subsection{Proof of Theorem~\ref{mainthm}}
Here we prove that the map $\Phi$ defined in the previous section is an isomorphism
of graded $\QQ$-algebras. Our approach has been inspired by \cite{ChenTseng}.

First of all, we prove the following result.
\begin{proposition} \label{isomorphism of vector spaces}
The linear map $\Phi \colon \Homol^*_{\rm CR}([\Hy^3_\C / \Gamma], \QQ) \to \Homol^* (Y, \QQ)$ defined in 
the previous section is an isomorphism of vector spaces.
\end{proposition}
\begin{proof} 
We use the Mayer-Vietoris  exact sequence for  Chen-Ruan orbifold cohomology.
We will define an appropriate open covering of the orbifold  $[\Hy^3_\C/\Gamma]$. This induces 
an open covering of the inertia orbifold. Since the Chen-Ruan orbifold cohomology is the 
usual cohomology of the inertia orbifold, we have a Mayer-Vietoris long exact sequence. 

The open covering is defined as follows.
As before, let $x_1, \ldots , x_s \in \Hy^3_\C/\Gamma$ be the singular points of the singular locus $\Sigma$
of $\Hy^3_\C/\Gamma$, and let $X: = (\Hy^3_\C/\Gamma ) \setminus \{ x_1, \ldots , x_s\}$.
Then there is a unique open sub-orbifold $\mathcal{X} \subset [\Hy^3_\C/\Gamma]$
having   $X$ as coarse moduli space. Notice that $X$ is an analytic space with transverse singularities of type A.
Let now, for any $i=1, \ldots , s$, $W_i \subset \Hy^3_\C/\Gamma$ be an open neighbourhood of 
$x_i$ isomorphic to $\tilde{W_i}/G_i$, where $\tilde{W_i}$ is an open subset of $\Hy^3_\C$
isomorphic to an open ball, and $G_i$ is the stabilizer of a point $z_i\in \tilde{W_i}$
that maps onto $x_i$ under the quotient map $\Hy^3_\C \to \Hy^3_\C/\Gamma$.
Without loss of generality, we suppose that $W_1, \ldots , W_s$ are pairwise 
disjoint. Then $\mathcal{W}:= \sqcup_{i=1}^s[\tilde{W_i}/G_i]$ is an open sub-orbifold 
of $[\Hy^3_\C/\Gamma]$. Let us denote with $W:= \sqcup_{i=1}^s W_i$ the coarse moduli space 
of $\mathcal{W}$. 

Let us consider the open covering  $\{ \mathcal{X}, \mathcal{W}\}$ of $[\Hy^3_\C/\Gamma]$,  
the open covering $\{ f^{-1}(X), f^{-1}(W)\}$ of $Y$ and the corresponding long exact cohomology sequences
of Mayer-Vietoris. By Lemma~\ref{Phisheaf},   $\Phi$ induces a morphism between long exact 
sequences as follows: \small
$$
\xymatrix{ 
\Homol^{k-1}(f^{-1}(X) \cap f^{-1}(W)) \ar[r] & \Homol^{k}(Y) \ar[r] &
\Homol^{k}(f^{-1}(X)) \oplus  \Homol^k(f^{-1}(W)) \ar[r] &\Homol^{k}(f^{-1}(X) \cap f^{-1}(W))  \\
 \Homol^{k-1}_{\rm CR}(\mathcal{X} \cap \mathcal{W}) \ar[u]^{\Phi^{X\cap W}} \ar[r] 
& \Homol^{k}_{\rm CR}([\Hy^3_\C/\Gamma]) \ar[u]^\Phi \ar[r] & 
\Homol^{k}_{\rm CR}(\mathcal{X}) \oplus \Homol^k_{\rm CR}(\mathcal{W}) \ar[u]^{\Phi^X \oplus \Phi^W} \ar[r] &
\Homol^{k}_{\rm CR}(\mathcal{X} \cap \mathcal{W}) \ar[u]^{\Phi^{X\cap W}}
}
$$ \normalsize
The map $\Phi^W$ is an isomorphism  by \cite{IR}. Therefore, the Proposition follows from the five-lemma
if $\Phi^X$ and $\Phi^{X\cap W}$ are isomorphisms. To see that they are isomorphisms, recall that  
$X$ and $X\cap W$ are analytic spaces with transverse singularities of type A. 
Therefore, $\Phi^X$ and $\Phi^{X\cap W}$ are isomorphisms if the monodromy is trivial (\cite[Prop. 4.8 and 4.9]{Perroni}).
On the other hand, if the monodromy is not trivial, then there is an unramified double covering 
$\tilde{\mathcal{X}} \to \mathcal{X}$ such that $\tilde{\mathcal{X}}$ has transverse singularities of type $A$
and trivial monodromy (Proposition~\ref{Bianchi orbifolds monodromy}). Let $\widetilde{f^{-1}(X)} \to \tilde{X}$ be the crepant resolution of $\tilde{X}$
(the coarse moduli space of $\tilde{\mathcal{X}}$). Then there is a natural map $\widetilde{f^{-1}(X)} \to f^{-1}(X)$,
which is an unramified double covering. Since $\Homol^{*}_{\rm CR}(\mathcal{X}) \cong \Homol^{*}_{\rm CR}(\tilde{\mathcal{X}})^{\ZZ/2\ZZ}$
(\cite[Prop. 3.13]{Perroni}) and $\Homol^*(f^{-1}(X)) \cong \Homol^*\left( \widetilde{f^{-1}(X)} \right)^{\ZZ/2\ZZ}$,
we conclude that $\Phi^X$ is an isomorphism.   The same proof works for $\Phi^{X\cap W}$.
\end{proof}

The proof of Theorem~\ref{mainthm} is now completed when combining Proposition~\ref{isomorphism of vector spaces} with the following statement.
\begin{proposition}
$\Phi \colon (\Homol^*_{\rm CR}([\Hy^3_\C / \Gamma], \QQ) , \cup_{\rm CR}) \to (\Homol^* (Y, \QQ), \cup)$ 
is a ring homomorphism.
\end{proposition}
\begin{proof} 
Notice that on the non-twisted sector, $\Phi$
preserves the cup products because 
$f^*$ is a ring homomorphism.  So let $\alpha_g , \beta_h$
be cohomology classes of the twisted sectors 
$(\Hy^3_{\C})^g / C_\Gamma (g), (\Hy^3_{\C})^h / C_\Gamma (h)$. Since ${\rm shift}(g)
={\rm shift}(h)=1$, the Chen--Ruan degrees $\deg (\alpha_g), \deg (\beta_h)$
are $\geq 2$, hence $\deg \left( \alpha_g \cup_{\rm CR} \beta_h \right) \geq 4$.
By Theorem~\ref{spine}, we conclude that 
$\Homol^d_{\rm CR}([\Hy^3_\C / \Gamma])=0$ if $d\geq 4$; so
 $\alpha_g \cup_{\rm CR} \beta_h =0$.
On the other hand, since $\Phi$ is grading preserving, 
$\deg \left( \Phi (\alpha_g ) \cup \Phi (\beta_h) \right) \geq 4$,
so also $\Phi (\alpha_g ) \cup \Phi (\beta_h) =0$ because 
$\Homol^d (Y) \cong \Homol^d_{\rm CR}([\Hy^3_\C / \Gamma])$, for any  $d$.
Finally, let $\alpha_g \in \Homol^*((\Hy^3_{\C})^g / C_\Gamma (g))$
and $\beta \in \Homol^*(\Hy^3_{\C}/\Gamma)$.
Then, $\alpha_g \cup_{\rm CR} \beta = \alpha_g \cup \imath_g^* \beta \in 
\Homol^*((\Hy^3_{\C})^g / C_\Gamma (g))$, so
$\Phi \left( \alpha_g \cup_{\rm CR} \beta \right) = (\jmath_g)_* \pi^* \left( \alpha_g \cup 
\imath_g^* \beta \right)$. On the other hand,
\begin{eqnarray*}
\Phi (\alpha_g) \cup \Phi (\beta) &=& (\jmath_g)_* \pi^* (\alpha_g) \cup f^* (\beta) \\
&=& (\jmath_g)_* \left( \pi^* (\alpha_g) \cup \jmath_g^*(f^* (\beta)) \right) \qquad 
\mbox{(projection formula)}\\
&=& (\jmath_g)_* \left( \pi^* (\alpha_g) \cup \pi^* (\imath_g^* (\beta)) \right) 
\qquad 
(f \circ \jmath_g = \imath_g \circ \pi) \\
&=& (\jmath_g)_* \pi^* \left( \alpha_g \cup \imath_g^* (\beta) \right) \\
&=& \Phi (\alpha_g \cup_{\rm CR} \beta) \, .
\end{eqnarray*}

\end{proof}

\section{Cohomological Crepant Resolution Conjecture
for Bianchi orbifolds}  \label{vanishing_section}

In this section, we compare the results obtained so far 
with the Cohomological Crepant Resolution Conjecture of Ruan.
We begin by briefly reviewing the statement of this conjecture,
 referring to \cite{RuanCCRC}, \cite{CoatesRuan},
and the references therein, 
for further details. 

Let $\sX$ be a complex orbifold, and let $X$ be its coarse moduli space.
We assume that $X$ is a complex projective variety which has a crepant
resolution $f\colon Y \to X$. The quantum corrected cohomology ring of 
$f\colon Y \to X$ is a ring structure on the vector space 
$\Homol^*(Y,\bC)=\oplus_{d \geq 0} \Homol^d(Y, \bC)$,
which is a deformation of the standard cohomology ring 
of $Y$. Its definition depends on the 
choice of a basis of $\ker \left( f_* \colon \Homol_2(Y,\bQ) \to \Homol_2(X,\bQ)\right)$
consisting of homology classes of effective curves $\beta_1, \ldots , \beta_n$. 
One defines the $3$-point function 
\begin{equation}\label{3pointsfunction}
(\alpha_1, \alpha_2, \alpha_3)(q_1, \ldots , q_n) 
=\sum_{(k_1, \ldots, k_n)\in \bN^n} \left\langle \alpha_1, \alpha_2, \alpha_3 
\right\rangle^Y_\beta q_1^{k_1}\cdot \ldots \cdot q_n^{k_n} \, ,
\end{equation}
where $\beta = k_1 \beta_1 + \ldots + k_n \beta_n \in \Homol_2(Y,\bZ)$,  and 
$\left\langle \alpha_1, \alpha_2, \alpha_3 
\right\rangle^Y_\beta$ is the Gromov-Witten invariant of $Y$, of genus $0$,
of homology class $\beta$, 
with respect to the cohomology classes $\alpha_1, \alpha_2, \alpha_3
\in \Homol^*(Y, \bC)$.
Recall that a compact complex curve $D\subset Y$ of homology class $\beta$ is called an 
{\it exceptional curve} for $f$.
To simplify the discussion, we assume that the $3$-point function
\eqref{3pointsfunction}
converges in a neighborhood of the origin $(q_1, \ldots, q_n)=(0,\ldots,0)$
(see \cite{CoatesRuan} for the general case);
and then, for any $(q_1, \ldots, q_n)$ in this neighborhood, we define 
a product $\star_f$ on the cohomology of $Y$ as follows:
Given cohomology classes $\alpha_1 , \alpha_2$, then
$\alpha_1 \star_f \alpha_2$ is the cohomology class which satisfies 
the following equation:
$$
(\alpha_1 \star_f \alpha_2 , \alpha_3) = 
(\alpha_1, \alpha_2, \alpha_3)(q_1, \ldots , q_n) \, ,
\quad \forall \alpha_3 \in \Homol^*(Y, \bC) \, ,
$$
where the pairing $(,)$ to the left hand side is the Poincar\'e pairing of $Y$.
The product $\star_f$ satisfies the usual properties 
of the cup product, e.g. it is associative, graded-commutative, and
$1$ is its neutral element. The family of rings 
$\left(  \Homol^*(Y, \bC), \star_f \right)$, as $(q_1, \ldots , q_n)$ varies,
is part of the (small) quantum cohomology of $Y$. 
Assigning to $q_1, \ldots , q_n$ specific values, we obtain 
the so-called 
quantum corrected cohomology ring of $f\colon Y \to X$ (\cite{RuanCCRC},
\cite{CoatesRuan}).
Notice that if  $(q_1, \ldots, q_n)=(0,\ldots,0)$, then $\star_f$  coincides with the 
usual cup product, as it  follows from the fact that  
$(\alpha_1, \alpha_2, \alpha_3)(0, \ldots , 0) = \int_Y \alpha_1\cup \alpha_2
\cup \alpha_3$. So, the quantum corrected cohomology ring of $f\colon Y \to X$
is regarded as a deformation of the usual cohomology ring of $Y$.

Ruan's Cohomological Crepant Resolution Conjecture predicts
that there is an analytic continuation of \eqref{3pointsfunction} 
to a region containing  a point $(\bar{q}_1, \ldots, \bar{q}_n)$
such that, for $({q_1}, \ldots, {q_n})=(\bar{q}_1, \ldots, \bar{q}_n)$,
the ring
$\left( \Homol^* (Y, \bC), \star_f \right)$ is isomorphic
to the Chen--Ruan orbifold cohomology ring
$\left( \Homol^*_{\rm CR}(\sX) , \cup_{\rm CR} \right)$ of $\sX$.

In the case of a Bianchi orbifold $[\Hy^3_\C / \Gamma]$,
the coarse moduli space
$\Hy^3_\C / \Gamma$ is not a projective variety \cite{ElstrodtGrunewaldMennicke},
and so, for every crepant resolution $f\colon Y \to \Hy^3_\C / \Gamma$,
$Y$ is not a projective variety.
Hence the  Gromov-Witten invariants of $Y$ are, in general, 
not well defined. However, we will see that $[\Hy^3_\C / \Gamma]$ 
has a K\"ahler structure, and that one does not expect non-zero quantum corrections
coming from exceptional curves for $f$. 
This is motivated by a deformation theoretic argument 
about the complex structure of $Y$ (let us recall that the Gromov-Witten invariants are invariant
under deformations of the complex structure).
More precisely, we conjecture that, for any homology class $\beta \in \Homol_2(Y,\mathbb{Z})$ 
of a connected exceptional curve for $f$, there is an open subset  
$\mathcal{U} \subset Y$ containing  all the connected curves 
$D\subset Y$ of homology class $\beta$,
and a deformation of the complex structure of $\mathcal{U}$ that does not contain 
any compact complex curve.

In this article we prove the latter conjecture in one special case, 
namely $\Gamma = \text{PSL}_2(\ringO_{-5})$,
while the general case should be feasible with similar arguments.
Hence, in accordance with Ruan's conjecture, there should be a ring isomorphism 
$\left( \Homol^*_{\rm CR}([\Hy^3_\C / \Gamma]) , \cup_{\rm CR} \right)
\cong \left( \Homol^* (Y), \cup \right)$.   This is confirmed 
by our Theorem~\ref{mainthm}.

\begin{proposition}
Let $[\Hy^3_\C / \Gamma]$ be a Bianchi orbifold. 
Then the Bergman metric on $\Hy^3_\C$
descends to a K\"ahler (orbifold) metric on $[\Hy^3_\C / \Gamma]$.
\end{proposition}
\begin{proof}
Let 
\begin{equation*}\label{Bergman}
\operatorname{ds}^2 = \sum g_{\alpha \bar{\beta}} 
\operatorname{d}z_\alpha \operatorname{d}\bar{z}_\beta
\end{equation*}
be the Bergman metric on $\Hy^3_\C$. By \cite[Theorem 8.4, p. 144]{MorrowKodaira},
$\operatorname{ds}^2$ is invariant under the action of $\Gamma$, hence it induces
a K\"ahler metric on the orbifold $[\Hy^3_\C / \Gamma]$.
\end{proof}

Let now $f\colon Y \to \Hy^3_\C / \Gamma$ be a crepant resolution.
Let $D\subset Y$ be an exceptional, 
 compact, complex and connected curve, that is $f_* ([D]) = 0$, where 
 $[D]$ is the fundamental class of $D$. Since 
$[\Hy^3_\C / \Gamma]$ is K\"ahler, $f( D )$ is a point,
so $D$ is contained in the exceptional divisor of $f$. 
In particular, for any homology class 
$\beta \in \ker \left( f_* \colon \Homol_2 ( Y, \bQ) \to 
\Homol_2 (\Hy^3_\C / \Gamma , \bQ) \right)$, and for any 
stable map $\mu \colon C \to Y$, such that $\mu_* ( [C]) = \beta$,
the image of $\mu$ is contained in the exceptional divisor of $f$.
Hence it suffices to consider the problem locally in a neighbourhood of the 
exceptional divisor.

From the results of Sections \ref{The conjugacy classes of finite order 
elements in the Bianchi groups} and 
\ref{Kraemer numbers and orbifold cohomology}, we see that 
the singular locus of $\Hy^3_\C / \Gamma$ is the union of 
several irreducible  components, each of which is isomorphic 
either to $\Delta = \{ z\in \C \, | \, |z|<1\}$ or to $\Delta^* = \Delta \setminus \{0\}$. Furthermore, the generic point of each  irreducible component
of the singular locus 
is a transverse singularity of type ${\rm A}_n$ of $\Hy^3_\C / \Gamma$, with $n=1$ or $2$.

Let us now consider the special case where
$\Gamma = \text{PSL}_2(\ringO_{-5})$ (see Section \ref{case -5}).
In this case we show that the quantum corrections 
to the cohomology ring of $Y$ coming from  exceptional curves  vanish. 
The singular locus of 
$X=\Hy^3_\C / \Gamma$ has two connected components,
$X_{(2)}\cong \Delta^*$, whose points are transverse singularities 
of type ${\rm A}_2$,  and $X_{(1)}$, that is the union of three irreducible components, 
$X_{(1)}', X_{(1)}'', X_{(1)}''' \cong \Delta$, that meet in two points $P, Q$
and the complement $X_{(1)} \setminus \{P,Q\}$ is a locus of transverse singularities
of type ${\rm A}_1$ (see Figure~\ref{singularities} and Section \ref{case -5}). 
The exceptional divisor of $f\colon Y \to X$ has two connected components:
$E_{(2)}$, which is mapped to $X_{(2)}$ by $f$, and $E_{(1)}$, such that 
$f(E_{(1)}) = X_{(1)}$. Furthermore, $E_{(1)}$ has three irreducible components,
$E_{(1)}', E_{(1)}'', E_{(1)}'''$, that are mapped by $f$ to $X_{(1)}', X_{(1)}'', X_{(1)}'''$,
respectively. 

Let us consider first the ${\rm A}_2$-singularities $X_{(2)}$. Notice that from the 
presentation of the Chen-Ruan cohomology (Section \ref{case -5}) it follows that 
$\Hy^3_\C/\Gamma$ has trivial monodromy on $X_{(2)}$. Hence 
there is an open neighborhood $U$ of $X_{(2)}$, such that 
$U\cong \tilde{U}/(\bZ/3\bZ)$, where $\tilde{U}$ is a complex manifold
with an action of $\bZ/3\bZ$, such that the fixed-points locus 
$\tilde{U}^{\bZ/3\bZ}$ is a smooth submanifold of $\tilde{U}$
isomorphic to $X_{(2)}$.
Furthermore, up to deformation, we can assume that $\tilde{U}$
is an open neighbourhood of the zero-section of the 
normal bundle $N_{\tilde{U}^{\bZ/3\bZ}|\tilde{U}}$
of $\tilde{U}^{\bZ/3\bZ}$ in $\tilde{U}$. 
This can be achieved using  the {\it deformation to the normal cone} of the imbedding 
$\tilde{U}^{\bZ/3\bZ}\subset \tilde{U}$ (\cite[Chapter 5]{Fulton}).
The vector bundle map $N_{\tilde{U}^{\bZ/3\bZ}|\tilde{U}} \to \tilde{U}^{\bZ/3\bZ} \cong X_{(2)}$
induces a morphism $\tilde{U}/(\bZ/3\bZ) \to X_{(2)}$ that equips $U\cong \tilde{U}/(\bZ/3\bZ)$ with the structure of a fibration over $X_{(2)}$,
with fibres all isomorphic to the surface singularity of type ${\rm A}_2$.
The important fact is that this fibration is trivial. To see this, 
let us recall that the action of $\bZ/3\bZ$ on the fibres of $N_{\tilde{U}^{\bZ/3\bZ}|\tilde{U}}$ 
induces a splitting,
$N_{\tilde{U}^{\bZ/3\bZ}|\tilde{U}} = \mathbb{L} \oplus \mathbb{M}$,
where $\mathbb{L}$ and $\mathbb{M}$ are   the eigenbundles 
corresponding to the irreducible characters of the representation of $\bZ/3\bZ$
on the fibres of $N_{\tilde{U}^{\bZ/3\bZ}|\tilde{U}}$.
In our case, $\mathbb{L}$ and $\mathbb{M}$ are trivial line bundles on $X_{(2)}$ (see \cite[Thm. 30.3, p. 229]{Forster}),
therefore the fibration $N_{\tilde{U}^{\bZ/3\bZ}|\tilde{U}}/({\bZ/3\bZ}) \to X_{(2)}$
is trivial, that is, it is isomorphic to the projection to the first factor
of $X_{(2)} \times \{(u,v,w)\in \bC^3 \, | \, uv=w^{3}\}$.
Now, using the theory of deformations of rational double points (see \cite{Brieskorn}, 
\cite{Tyurina}), we  deform 
 the family $N_{\tilde{U}^{\bZ/3\bZ}|\tilde{U}}/({\bZ/3\bZ}) \to X_{(2)}$
to a family of affine smooth surfaces.
Finally, consider the neighbourhood  $\mathcal{U} := f^{-1}(U)$ of $E_{(2)}$.
Taking a simultaneous resolution of the
previous deformation of $N_{\tilde{U}^{\bZ/3\bZ}|\tilde{U}}/({\bZ/3\bZ}) \to X_{(2)}$, 
we obtain a deformation of $\mathcal{U}$ to a manifold 
that does not contain compact complex curves. 

Let us now consider the exceptional curves that are contained in $E_{(1)}$.
Notice that each component $E_{(1)}'$, $E_{(1)}''$, $E_{(1)}'''$, can be 
seen as the exceptional divisor of a crepant resolution of a transverse
singularity of type ${\rm A}_1$. Hence, from our description of the obstruction 
bundles (Theorem~\ref{obstructionBundles}) and from \cite[Theorem 7.6]{Perroni},
it follows that the exceptional curves contained in one of these components 
do not contribute to the quantum corrected cohomology ring of $Y$.
If $D\subset Y$ is a connected exceptional curve which is contained 
in more than one component of $E_{(1)}$, then $f(D)$ coincides with $P$ or $Q$,
the points where the components $X_{(1)}', X_{(1)}'', X_{(1)}'''$ meet together. 
Near $P$ and $Q$, $X$ is isomorphic to the singularity $\bC^3/ \Kleinfourgroup$
(see Section \ref{case -5}), where $\Kleinfourgroup =
\left\langle \xi, \eta \, | \, \xi^2 = \eta^2 = (\xi \eta )^2 =1 \right\rangle \cong \bZ/2\bZ \oplus \bZ / 2\bZ$.
We can realize the quotient    $\bC^3/ \Kleinfourgroup$ as $\left( \bC^3/ \left\langle \xi \right\rangle \right)/
\left\langle \eta \right\rangle$, 
and notice that $\bC^3/ \left\langle \xi \right\rangle \cong \{ (u,v,w,z)\in \bC^3\times \bC \, | \, uv=w^2\}$
with the action of $\left\langle \eta \right\rangle$ given by $\eta \cdot (u,v,w,z)\mapsto (u,v,-w,-z)$.
The semi-universal deformation of $\{ (u,v,w)\in \bC^3 \, | \, uv=w^2\}$
is $uv=w^2 +t$, where $t$ is the deformation parameter. 
Notice that the action of $\left\langle \eta \right\rangle$ on $\bC^3/ \left\langle \xi \right\rangle$ 
extends to $\{ (u,v,w,z)\in \bC^3\times \bC \, | \, uv=w^2 +t\}$, for all $t$,
as follows: $\eta \cdot (u,v,w,z) = (u,v,-w,-z)$.
Hence $\{ (u,v,w,z)\in \bC^3\times \bC \, | \, uv=w^2 +t\}/\left\langle \eta \right\rangle$, for $t\in \bC$,
is a deformation
of $\bC^3 / \Kleinfourgroup$.
Notice that for $t\not= 0$, $\{ (u,v,w,z)\in \bC^3\times \bC \, | \, uv=w^2 +t\}/\left\langle \eta \right\rangle$
has transverse singularities of type ${\rm A}_1$, and they can be 
smoothed by a deformation as follows.  
Taking the invariants of the $\left\langle \eta \right\rangle$-action, we see that
$$
\{ (u,v,w,z)\in \bC^3\times \bC \, | \, uv=w^2 +t\}/\left\langle \eta \right\rangle \cong
\{ (u,v,\rho, \sigma, \tau )\in \bC^5 \, | \, uv=\rho + t \, , \, \rho \sigma = \tau^2\} \, ,
$$
where $\rho = w^2, \sigma = z^2, \tau = wz$. And so, 
$\{ (u,v,\rho, \sigma, \tau )\in \bC^5 \, | \, uv=\rho + t \, , \, \rho \sigma = \tau^2 +s\}$
is a deformation of $\bC^3 / \Kleinfourgroup$, with deformation parameters $t$ and $s$.
For $t\not= 0$ and $s\not= 0$, the variety 
$\{ (u,v,\rho, \sigma, \tau )\in \bC^5 \, | \, uv=\rho + t \, , \, \rho \sigma = \tau^2 +s\}$
is an affine smooth variety. Thus, a simultaneous resolution of this family 
yields a deformation of a neighbourhood of $f^{-1}( P )$ ($f^{-1}( Q )$, respectively),
such that the generic member of the family is a smooth affine variety.
Hence, it does not contain compact complex curves.

\section{Orbifold cohomology computations for sample Bianchi orbifolds} \label{Sample orbifold cohomology computations for the Bianchi groups}

\begin{wrapfigure}[6]{r}{54mm}
  \centering\scalebox{0.95}{
  \includegraphics[width=14mm]{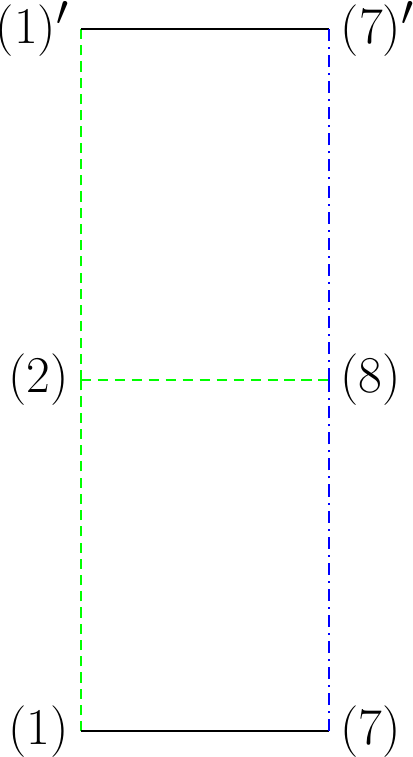}
  }
  \caption{Fundamental domain in the case \mbox{$m = 2$}.} \label{m2}
\end{wrapfigure}

We will carry out our computations in the upper-half space model
$ \left\{ x +iy +rj \in \C \oplus \R j \suchthat r > 0 \right\}$
for~$\Hy^3_\R$ in three cases.
Details on how to compute Chen--Ruan orbifold cohomology can be found in \cite{Perroni}.
In the case $\Gamma = \text{PSL}_2(\mathbb{Z}[\sqrt{-5}\thinspace])$,
we also compute the cohomology ring structure.

\subsection{The case $\Gamma = \text{PSL}_2(\mathbb{Z}[\sqrt{-2}\thinspace])$.}

Let $\omega := \sqrt{-2}\thinspace$. A fundamental domain for $\Gamma := \text{PSL}_2(\mathbb{Z}[\omega])$ in real hyperbolic 3-space~$\Hy$ has been found by Luigi Bianchi~\cite{Bianchi}.
We can obtain it by taking the geodesic convex envelope  of its lower boundary (half of which is depicted in Figure~\ref{m2}) and the vertex $\infty$,  and then removing the vertex $\infty$, making it non-compact. The other half of the lower boundary consists of one isometric $\Gamma$--image of each of the depicted 2-cells (in fact, the depicted 2-cells are a fundamental domain for a $\Gamma$--equivariant retract of $\Hy$, which is described in \cite{RahmFuchs}). The coordinates of the vertices of Figure~\ref{m2} in the upper-half space model are
$(1) = j$,
$(1)' = \omega +j$, \\
$(2) = \frac{1}{2}\omega +\sqrt{\frac{1}{2}}j$, 
$(7) = \frac{1}{2}  +\sqrt{\frac{3}{4}}j$,
$(7)' = \frac{1}{2}  +\omega +\sqrt{\frac{3}{4}}j$, 
$(8) = \frac{1}{2} +\frac{1}{2}\omega +\frac{1}{2}j$.

The 2-torsion sub-complex 
(dashed) and the 3-torsion sub-complex
(dotted) are indicated in the figure. 
The set of representatives of conjugacy classes can be chosen 
$$T = \{ \one,\thinspace \alpha,\thinspace \gamma,\thinspace \beta,\thinspace \beta^2 \},$$
with 
$\alpha = \pm $ \scriptsize $\begin{pmatrix}1 & \omega \\ \omega & -1\end{pmatrix}$, \normalsize 
$\beta = \pm $ \scriptsize $\begin{pmatrix}0 & -1\\1 & 1\end{pmatrix}$ \normalsize and
$\gamma = $ \scriptsize $\pm \begin{pmatrix}0 & 1\\-1 & 0\end{pmatrix}$, \normalsize \thinspace
so $\alpha$ and $\gamma$ are of order~2, and $\beta$ is of order~3. 
Using Lemma~\ref{numberOfConjugacyClasses} and with the help of our Bredon homology computations, we check the cardinality of~$T$.
The fixed point sets are then the following subsets of complex hyperbolic space $\Hy := \Hy^3_\C$:
\\
$\mathcal{H}^\one = \mathcal{H}$,
\\
$\mathcal{H}^\alpha = $ the complex geodesic line through $(2)$ and $(8)$,
\\
$\mathcal{H}^\gamma = $ the complex geodesic line through $(1)$ and $(2)$,
\\
$\mathcal{H}^\beta = \mathcal{H}^{\beta^2} = $ the complex geodesic line through $(7)$ and $(8)$.

The matrix $g =$ \scriptsize  $\pm \begin{pmatrix}1 & -\omega\\ 0 & 1\end{pmatrix}$ \normalsize
 performs a translation preserving the $j$-coordinate and sends the edge $(1)(7)$ onto the edge $(1)'(7)'$,
 so the orbit space $\Hy_\R /_\Gamma$ is homotopy equivalent to a circle.
Consider the real geodesic line $\mathcal{H}_\R^\gamma$ on the unit circle of real part zero.
The edge $g^{-1}\cdot \left((2)(1)'\right)$ = $\left(g^{-1}(2)\right)(1)$ lies on $\mathcal{H}_\R^\gamma$ and is not $\Gamma$--equivalent to the edge $(1)(2)$.
Because of Lemma~\ref{reflection}, the centralizer $C_\Gamma(\gamma) $ reflects the line  $\mathcal{H}_\R^\gamma$ onto itself at $(2)$, and again at $g^{-1}(2)$.
Furthermore, none of the four elements of~$\Gamma$ sending $(2)$ to $g^{-1}(2)$ belongs to $C_\Gamma(\gamma) $ .
Hence the quotient space $\mathcal{H}_\R^\gamma /_{C_\Gamma(\gamma)} $ consists of a contractible segment of two adjacent edges. 
Thus \mbox{$\Homol^{d-2} \left(\mathcal{H}_\C^\gamma  /_{C_\Gamma(\gamma)}; \thinspace \rationals \right) \cong $
\scriptsize$\begin{cases}\rationals, & d=2 \\ 0 & \mathrm{else} \end{cases}$\normalsize}
is contributed to the orbifold cohomology.

Next, consider the real geodesic line $\mathcal{H}_\R^\beta$ on the circle of constant real coordinate $\frac{1}{2}$, of center $\frac{1}{2}$ and radius $\sqrt{\frac{3}{4}}$.
The edge $g^{-1}\cdot \left((8)(7)'\right)$ = $\left(g^{-1}(8)\right)(7)$ lies on $\mathcal{H}_\R^\beta$ and is not $\Gamma$--equivalent to the edge $(7)(8)$.
The centralizer of $\beta$ contains the matrix
 $V:= \pm $ \scriptsize $\begin{pmatrix}2 1-\omega \\ \omega -1 & 1+\omega \end{pmatrix}$ \normalsize \thinspace
of infinite order, which sends  the edge $\left(g^{-1}(8)\right)(7)$ to $(8)z$ with $z = \frac{1}{2} +\frac{3}{5}\omega +\sqrt{\frac{3}{100}}j $. 
We conclude that the translation action of the group $\left\langle V \right\rangle$ on the line $\mathcal{H}_\R^\beta$ is transitive, with quotient space represented by the circle  $\left(g^{-1}(8)\right)(7) \cup (7)(8)$, first and last vertex identified.
Thus \mbox{$\Homol^{d-2} \left(\mathcal{H}_\C^\beta  / _{C_\Gamma(\beta)}; \thinspace \rationals \right) \cong
\Homol^{d-2} \left(\mathcal{H}_\C^{\beta^2}  /_{C_\Gamma(\beta^2)}; \thinspace \rationals \right) \cong$
\scriptsize$\begin{cases}\rationals, & d=2,3 \\ 0 & \mathrm{else} \end{cases}$\normalsize}
is contributed to the orbifold cohomology.

Because of Lemma~\ref{reflection}, the centralizer $C_\Gamma(\alpha) $ reflects the line  $\mathcal{H}_\R^\alpha$ onto itself at $(2)$, and again at $(8)$.
So, the quotient space $\mathcal{H}_\R^\alpha /_{C_\Gamma(\alpha)}$ is represented by the single contractible edge $(2)(8)$.
This yields that \mbox{$\Homol^{d-2} \left(\mathcal{H}_\C^\alpha  /_{C_\Gamma(\alpha)}; \thinspace \rationals \right) \cong$
\scriptsize$ \begin{cases}\rationals, & d=2 \\ 0 & \mathrm{else} \end{cases}$\normalsize}
is contributed to the orbifold cohomology.

Summing up over $T$, we obtain
$$ \Homol^d_{\rm CR}\left([\Hy^3_\C / \text{PSL}_2(\mathbb{Z}[\sqrt{-2}\thinspace])] \right) \cong 
\Homol^d\left(\Hy_\C/_{\text{PSL}_2(\mathbb{Z}[\sqrt{-2}\thinspace])}; \thinspace \rationals \right) \oplus
\begin{cases} \rationals^4, & d=2 ,\\ 
\rationals^2, & d=3, \\
 0, & \mathrm{otherwise}. \end{cases}$$

\subsection{The case $\Gamma = \text{PSL}_2(\ringO_{-5})$.}\label{case -5}
 
We start by analysing the case where  
$$
\Gamma = {\rm PSL}_2 (\Oh_{-5}) \, .
$$
In this case the singular locus of $X$ has two connected components.
One component is a transverse  singularity of type $A_2$ (we write t$A_2$).
The other component, drawn in Figure~\ref{singularities},
contains two singular points $P, Q$ which are analytically
isomorphic to the singularity at the origin of $\CC^3/\Kleinfourgroup$, where 
$$
\Kleinfourgroup = \left\langle \xi, \eta \, | \, \xi^2 = \eta^2 = (\xi \eta )^2 =1 \right\rangle \cong \ZZ/2 \oplus \ZZ/2
$$ 
is the Klein-four-group acting via the standard diagonal representation 
$\Kleinfourgroup \ra {\rm SL}_3(\CC)$:
$$
\xi \mapsto {\rm diag}(-1,-1,1) \, , \quad \,  \eta \mapsto {\rm diag}(-1,1,-1) \, . 
$$
The points $P, Q$ are joined by three curves of transverse singularities
of type $A_1$ (t$A_1$), which correspond in a neighbourhood of $P$
(resp. $Q$) to the image in $\CC^3/\Kleinfourgroup$ of the coordinate axes of $\CC^3$.

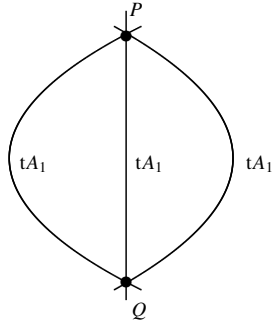
\begin{figure}
\begin{center}
 \scalebox{0.7}
{
\setlength{\unitlength}{1cm}
\begin{picture}(6,3)
\thicklines
\qbezier(3,-3)(7.5,-0.5)(3,2)
\qbezier(3.5,-3)(-1.5,-0.5)(3.5,2)
\put(3.22,-2.85){\circle*{0,2}}
\put(3.22,1.83){\circle*{0,2}}
\put(3.22,-3.2){\line(0,1){5.5}}
\put(3.26,2.2){\makebox{$P$}}
\put(3.3,-3.5){\makebox{$Q$}}
\put(1.2,-0.7){\makebox{t$A_1$}}
\put(3.4,-0.7){\makebox{t$A_1$}}
\put(5.5,-0.7){\makebox{t$A_1$}}
\end{picture}
}

\vspace{2.5cm}

 \caption{Two singular points $P, Q$ which are analytically
isomorphic to the singularity at the origin of $\CC^3/\Kleinfourgroup$.}
\label{singularities}
\end{center}
\end{figure}

From Corollary~\ref{introduced result}, we get the following presentation of 
the Chen--Ruan cohomology: 
$$
H^d_{\rm CR}([\HHH_\CC^3 /{\rm PSL}_2 (\Oh_{-5}) ] \, , \, \QQ) 
\cong 
H^d(\HHH_\CC^3 /{\rm PSL}_2 (\Oh_{-5}) , \QQ) \oplus
\begin{cases}
\QQ^2 \oplus \QQ^3 & d=2 \\
\QQ^2 \oplus \{0\} & d=3
\end{cases}
$$
where the first direct summand is the cohomology of the non-twisted sector.
The second direct summand $\left(\begin{matrix} \QQ^2 \\ \QQ^2 \end{matrix}
\right)$ 
is  the cohomology of the 
$3$-torsion twisted sector $\sX_{(3)}$ whose coarse moduli space is the connected
component of the singular locus of $X$ corresponding to the t$A_2$-singularity. 
Notice that this locus is topologically isomorphic to 
$S^1\times \RR \cong \CC^*$, $\la_6=1$ and $\la_6^* =0$,
where $\la_{2n}, \la_{2n}^*$ are as defined in Corollary~\ref{introduced result}. 
Finally, the third direct summand 
$\left(\begin{matrix} \QQ^3 \\ \{0 \} \end{matrix}
\right)$ 
is the cohomology of the $2$-torsion twisted sector $\sX_{(2)}$. 
This sector has three connected components each one 
homeomorphic to the strip $[0,1]\times \RR$ and 
corresponding to the t$A_1$-singularities joining the points
$P$ and $Q$ in Figure~\ref{singularities}. In the coarse moduli space $X$,
these components  form the configuration in Figure~\ref{singularities}.
 Here, we have $\la_4=\la_4^*=3$.

Now we study the Chen--Ruan cup product $\cup_{\rm CR}$,
verifying first that the ordinary cup product on the non-twisted sector 
$H^*(\HHH_\CC^3 /\Gamma , \QQ)$ vanishes.
From the explicit description of the quotient space
$\HHH_\RR^3 /{\rm PSL}_2 (\Oh_{-5})$ in~\cite{RahmFuchs}, 
we  get the picture of the Borel--Serre compactification of $\HHH_\RR^3 /{\rm PSL}_2 (\Oh_{-5})$
drawn in Figure~\ref{fundamental_domain}.
Here, we have expanded the singular cusp at $\frac{\sqrt{-5}+1}{2}$
to a fundamental rectangle $(s, s', s'', s''')$ for the action of the cusp stabilizer
$\Gamma_\frac{\sqrt{-5}+1}{2}$ on the plane attached by the Borel--Serre bordification.
In the same way, 
we expand the cusp at infinity to a fundamental rectangle 
$(\infty, \infty', \infty'', \infty''')$ 
for the action of the cusp stabilizer
$\Gamma_\infty$ on the plane attached there.
This is not visible in our $2$-dimensional diagram,
but is located above the rectangle $(o, o', o'', o''')$,
where $o$ is of height $1$ one above the cusp $0$.
The fundamental polyhedron for the $\Gamma$-action is then spanned by the rectangle 
$(\infty, \infty', \infty'', \infty''')$ 
and the polygons of Figure~\ref{fundamental_domain}. 
The face identifications of the fundamental polyhedron are
\begin{eqnarray}
(\infty, o, t, o',\infty') \sim (\infty''', o''', t', o'',\infty''),\\
  (\infty, o, b, u, o''',\infty''') \sim (\infty', o', b', u', o'',\infty''),\\
  (a''', s''', s'', a'', v') \sim (a, s, s', a', v),\\
  (u, a''', s''', s, a, b) \sim (u', a'', s'', s', a', b'),\\
  (o, t, v, a, b) \sim (o', t, v, a', b'),\\
  (o''', t', v', a''', u) \sim (o'', t', v', a'', u').
  \end{eqnarray}
Here, we did not respect the orientation of the $2$-cells,
but have written them in the way in which their vertices are identified.

It is well known that the Borel--Serre compactification of 
$\HHH_\RR^3 /{\rm PSL}_2 (\Oh_{-m})$ is homotopy equivalent to 
$\HHH_\RR^3 /{\rm PSL}_2 (\Oh_{-m})$ itself,
and it has been worked out in \cite{RahmNote}
how the boundary is attached in the compactification.
\\
\begin{figure}
  \centering
  \includegraphics[width=40mm]{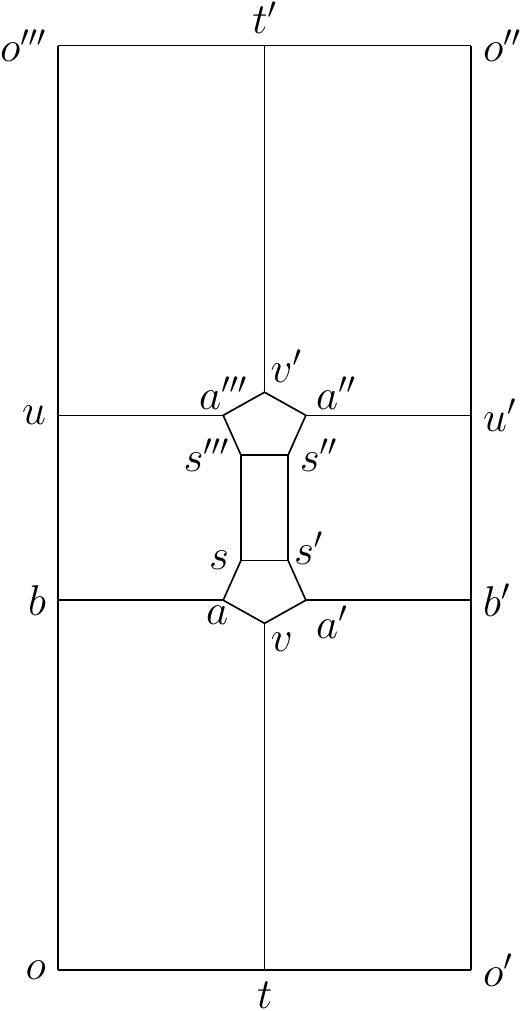}
    \caption{Fundamental domain for the Borel-Serre compactification in the case \mbox{$m = 5$}.}
    \label{fundamental_domain}
\end{figure} 
So we can describe the cohomology cocycles of $\HHH_\RR^3 /{\rm PSL}_2 (\Oh_{-5})$
in terms of the above fundamental polyhedron and face identifications.
By \cite{Goldman}[section 9.3], $\HHH_\CC^3$ admits a fundamental polyhedron $P_\CC$ for $\Gamma$ 
with the interior of its top-dimensional facets (called \emph{sides})
being open smooth submanifolds.
This yields a $\Gamma$-equivariant cell structure on $\HHH_\CC^3$.
The natural map $\HHH_\RR^3 \hookrightarrow \HHH_\CC^3 \to \HHH_\RR^3$ 
induces a map of the sides with respect to the fundamental polyhedron $P_\RR$ 
for $\Gamma$ on $\HHH_\CC^3$,
\\ ${}$ \hfill $ \text{sides}(P_\RR) \hookrightarrow \text{sides}(P_\CC) \to \text{sides}(P_\RR) ,$ \hfill  ${}$ \\
which respects the side identifications (\emph{side pairings}).
All of the side pairings of $P_\CC$ are detected this way, because they generate the group $\Gamma$ 
(see \cite{Goldman}[section 9.3]),
and so do already the side pairings of $P_\RR$.
Hence there are no additional identifications when complexifying the orbifold,
and thus there are no additional cohomology cocycles on $\HHH_\CC^3 /{\rm PSL}_2 (\Oh_{-5})$.
Generators for $H^1(\HHH_\RR^3 /{\rm PSL}_2 (\Oh_{-5}), \QQ)$
are, with reference to the above numbering of the identifications,
obtained from 
\begin{center}
$(\infty,\infty''')$ under (1) \qquad and $(s,s''')$ under (3).
\end{center}
Both $(\infty,\infty')$ under (2) and $(s,s')$ under (4)
yield trivial cocycles because of the identifications (5) and (6).

For instance using the arc method introduced in~\cite{Hatcher}[section 3.2],
we can now check explicitly that the cup product of the two cocycles obtained from 
$(\infty,\infty''')$ under (1) \qquad and $(s,s''')$ under (3) vanishes.
As further to the two $1$-dimensional cocycles, $\HHH_\RR^3 /\Gamma$ 
only admits a $0$- and a $2$-dimensional cocycle, and as there are no further identifications when complexifying, 
we arrive at the claimed vanishing of the ordinary cup product on the non-twisted sector 
$H^*(\HHH_\CC^3 /\Gamma , \QQ)$.

The cup product of two classes coming from the 
twisted sectors would be a class in dimension $\geq 4$,
where the twisted sectors vanish, 
and by the above calculation, so does the non-twisted sector.

Therefore, the Chen--Ruan cup product $\cup_{\rm CR}$
is trivial on $[\HHH_\CC^3 /{\rm PSL}_2 (\Oh_{-5})]$.

\subsection{The case $\Gamma = \text{PSL}_2(\mathbb{Z}[\sqrt{-1}\thinspace])$.}
\label{GaussExample}

 \begin{figure}
  \centering
\scalebox{0.65} 
{ \includegraphics{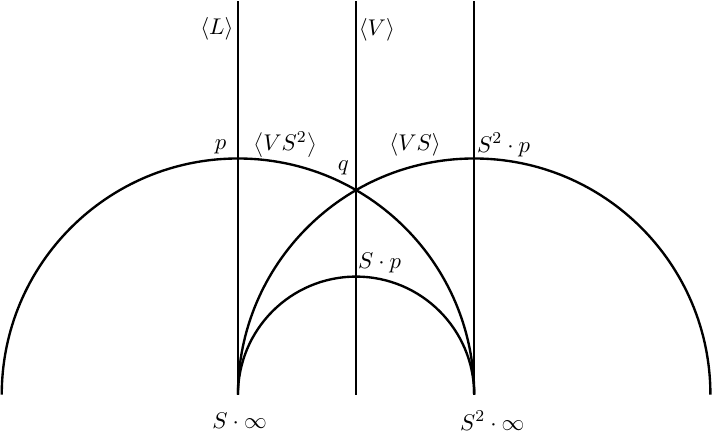}
%
%
%
%
%
%
%
%
%
%
}
  \caption{Geodesics fixed by certain finite order elements of $\text{PSL}_2(\mathbb{Z}[\sqrt{-1}\thinspace])$.} \label{GaussscherQuerschnitt}
\end{figure}

 \begin{figure}
  \centering
 \scalebox{0.85} 
{\includegraphics{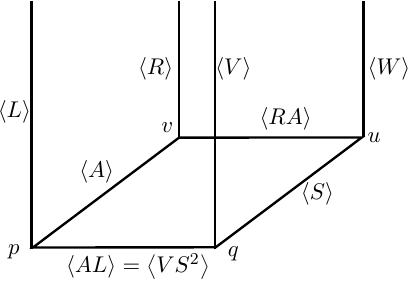}
}
  \caption{Half of a fundamental domain for the action of  $\text{PSL}_2(\mathbb{Z}[\sqrt{-1}\thinspace])$ on $\Hy$,
  open towards the cusp at $\infty$.
  The second half can be obtained as a copy of the open pyramid glued from below to its base square.} \label{m1}
\end{figure}

Let $i := \sqrt{-1}\thinspace$. A fundamental domain for the action of
$\Gamma := \text{PSL}_2(\mathbb{Z}[i])$ 
on real hyperbolic 3-space~$\Hy$ has been found by Luigi Bianchi,
and the stabilizers have been computed by Fl\"oge~\cite{Floege}, 
whose notation we are going to adopt.
It is drawn in Figure~\ref{m1}.
Here the vertex stabilizers are \begin{center}
$\Gamma_p = \left\langle A, L \right\rangle \cong \Kleinfourgroup$,
$\Gamma_q = \left\langle V, S \right\rangle \cong \Sthree$,
$\Gamma_u = \left\langle W, S \right\rangle \cong \Afour$,
$\Gamma_v = \left\langle R, A \right\rangle \cong \Sthree$,                                
                                \\
                                where              
$A = $ \scriptsize $\pm \begin{pmatrix}0 & 1\\-1 & 0\end{pmatrix}$, \normalsize
$L = \pm $ \scriptsize $\begin{pmatrix}-i & 0 \\ 0 & i\end{pmatrix}$, \normalsize 
$S = \pm $ \scriptsize $\begin{pmatrix}0 & -1\\1 & 1\end{pmatrix}$, \normalsize
\\
$R = \pm $ \scriptsize $\begin{pmatrix}-i & 1 \\ 0 & i\end{pmatrix}$, \normalsize 
$V = \pm $ \scriptsize $\begin{pmatrix}-i & -i \\ 0 & i\end{pmatrix}$\normalsize  and 
$W = \pm $ \scriptsize $\begin{pmatrix}-i & 1-i \\ 0 & i\end{pmatrix}$\normalsize ;
\\ $ 1 = A^2 = L^2 = V^2 = R^2 = S^3 = W^2.$
\end{center}

The matrices mentioned in Figure~\ref{m1} (and their square when they are of order $3$)
constitute a system of representatives modulo $\Gamma$ 
of the non-trivial elements of finite order.
So we compute the respective quotients of their rotation axis by their centralizer, 
in order to obtain the CR orbifold cohomology.
For the elements of order $3$, namely $RA$ and $S$, Theorem~\ref{3-torsion quotients}
and its proof pass unchanged, so
$\Hy^{RA}/_{C_\Gamma(\left\langle RA \right\rangle )} \cong \circlegraph$ and
$\Hy^{S}/_{C_\Gamma(\left\langle S \right\rangle )} \cong \circlegraph$.

For the elements of order $2$, we study the quotient of their fixed geodesic by their centralizer through Figure~\ref{GaussscherQuerschnitt}.
Further, we obtain another such figure useful for our purpose by making the following replacements on Figure~\ref{GaussscherQuerschnitt}:
$q \mapsto v$, $S \mapsto (RA)^2$, $VS^2 \mapsto A$, $V \mapsto R$.
The symmetries obtained from combining complex conjugation with the rotation by $L$ ensure that the relabeled figure is isometric to the printed one.

The points $p$, $S \cdot p$, $S^2 \cdot p$, $(RA)^2 \cdot p$, $R \cdot p$ all have stabilizer type $\Kleinfourgroup$,
because they are on the orbit of $p$, and hence the $2$-torsion axes passing through them are mirrored by
order-$2$-elements commuting with the rotation around the respective axis.
We immediately conclude that $\Hy^L/_{C_\Gamma(\left\langle L \right\rangle )}$ is represented by the half-open interval
$[p, \infty)$. In the stabilizer of $q$, which is of type $\Sthree$, 
apart from the trivial element, only the order $3$ element and its square commute with each other.
So there are no mirrorings at $q$ in the centralizer of the rotations with axis passing through $q$.
Hence, $\Hy^V/_{C_\Gamma(\left\langle V \right\rangle )} \cong [S\cdot p, q, \infty)$
and $\Hy^{VS^2}/_{C_\Gamma(\left\langle VS^2 \right\rangle )} \cong [p, q, S^2 \cdot \infty)$.

By the above described replacements on Figure~\ref{GaussscherQuerschnitt},
we obtain analogously that \\
$\Hy^R/_{C_\Gamma(\left\langle R \right\rangle )} \cong [(RA)^2 \cdot p, v, \infty)$
and
$\Hy^A/_{C_\Gamma(\left\langle A \right\rangle )} \cong [p, v, RA \cdot \infty)$.

In the stabilizer of the point $u$, there are order-$2$-elements commuting with $W$,
and therefore \\
$\Hy^W/_{C_\Gamma(\left\langle W \right\rangle )} \cong [u, \infty)$.

Summing up, and taking into account that $\Hy/\Gamma$ is contractible,
we obtain the CR orbifold cohomology

$$ \Homol^d_{\rm CR}\left([\Hy^3_\C / \text{PSL}_2(\mathbb{Z}[\sqrt{-1}\thinspace])] \right) \cong 
\begin{cases} 
\rationals, & d = 0, \\
\rationals^{10}, & d=2 ,\\ 
\rationals^4, & d=3, \\
 0, & \mathrm{otherwise}. \end{cases}$$

\subsection{The case $\Gamma = \text{PSL}_2(\ringO_{-3})$.} 
\label{EisensteinExample}
 
 \begin{figure} \centering
 \scalebox{0.7}
{ \includegraphics{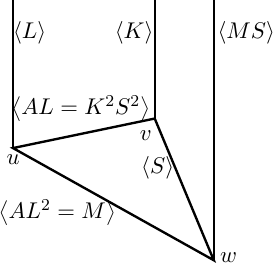}
%
%
%
%
%
%
%
%
%
}
  \caption{Half of a fundamental domain for the action of  $\text{PSL}_2(\ringO_{-3})$ on $\Hy$,
  open towards the cusp at $\infty$.
  The second half can be obtained as a copy of the open pyramid glued from below to its base square.} 
  \label{m3}
\end{figure}

 \begin{figure} \centering
\scalebox{1.0}
{ \includegraphics{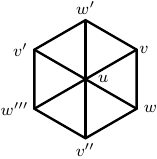}
%
}
  \caption{The $2$-cells in $\Hy$ equidistant to the cusps at $0$ and $\infty$,
  with no other $ \text{PSL}_2(\ringO_{-3})$-cusp being closer. The triangle $(u,v,w)$
  is the same one as in Figure~\ref{m3}, and the vertex $u$ sits on the middle of the geodesic $(0,\infty)$. }
\label{hexagon}
  \end{figure}
  
Let $\omega := \frac{\sqrt{-3} -1}{2}\thinspace$. 
A fundamental domain for the action of
$\Gamma := \text{PSL}_2(\mathbb{Z}[\omega])$ 
on real hyperbolic 3-space~$\Hy$ has been found by Luigi Bianchi,
and the stabilizers have been computed by Fl\"oge~\cite{Floege}, 
whose notation we are going to adopt.
It is drawn in Figure~\ref{m3}.
Here the vertex stabilizers are \begin{center}
$\Gamma_u = \left\langle A, L \right\rangle \cong \Sthree$,
$\Gamma_v = \left\langle K, S \right\rangle \cong \Afour$,      
$\Gamma_w = \left\langle M, S \right\rangle \cong \Afour$, 
                                \\
                                where              
$A = $ \scriptsize $\pm \begin{pmatrix}0 & 1\\-1 & 0\end{pmatrix}$, \normalsize
$L = \pm $ \scriptsize $\begin{pmatrix}-\omega^2 & 0 \\ 0 & \omega\end{pmatrix}$, \normalsize 
$S = \pm $ \scriptsize $\begin{pmatrix}0 & -1\\1 & 1\end{pmatrix}$, \normalsize
$K = \pm $ \scriptsize $\begin{pmatrix}\omega^2 & -\omega \\ 0 & \omega \end{pmatrix}$\normalsize ;
\\ $1 = A^2 = L^3 = K^3 = S^3 = M^2$. \normalsize 
\end{center} 

 As $\left\langle S \right\rangle \cong \Z/3$ and $\Gamma_v \cong \Afour \cong \Gamma_w$,
 the latter two vertex stabilizers do neither reflect $\Hy^S$,
 nor do they contribute any element to $C_\Gamma(\left\langle S \right\rangle)$. 
 That is why though all cusps are on one $\Gamma$-orbit, the centralizer
 $C_\Gamma(\left\langle S \right\rangle) \cong \left\langle S \right\rangle$
 leaves pointwise fixed $\Hy^S$, which is the geodesic line through $(v,w)$
 starting at a cusp $s$ in the $\Gamma_v$-orbit of $\infty$ 
 and ending at a cusp $e$ in the $\Gamma_w$-orbit of $\infty$. 
 By the $\Afour$-symmetries in $v$ and $w$,
 $(s,v)$ is mapped to $(\infty,v)$ and
 $(e,w)$ is mapped to $(\infty,w)$.
 Hence there can be no translations of $\Hy^S$ in $\Gamma$,
 and therefore $\Hy^S/_{C_\Gamma(\left\langle S \right\rangle)} = \Hy^S$.
 The $\Afour$-symmetries enforce all $3$-torsion axes passing through a representative of $v$ or $w$
 to admit the same centralizer quotient.
 Hence also $\Hy^K/_{C_\Gamma(\left\langle K \right\rangle)} = \Hy^K$
 and $\Hy^{MS}/_{C_\Gamma(\left\langle {MS} \right\rangle)} = \Hy^{MS}$
 are open geodesic lines starting and ending at cusps.
 \\
 In contrast, $\Hy^L$ is getting reflected onto itself by $\Gamma_u$. 
 But the elements of order $2$ in $\Gamma_u \cong \Sthree$ 
 do not commute with $L$, and hence 
 $\Hy^L = \Hy^L/_{C_\Gamma(\left\langle {L} \right\rangle)}$ is the geodesic line
 $(\infty, M \cdot \infty)$ with the vertex $u$ on its middle.
 
 Concerning the $2$-torsion axes, $\Hy^M$ does not get reflected by $\Gamma_u \cong \Sthree$.
 It gets reflected by order-$2$-elements in $\Gamma_w$ and $\Gamma_{v'}$
 commuting with $M$ (see Figure~\ref{hexagon});
 hence $\Hy^{M}/_{C_\Gamma(\left\langle M \right\rangle)} \cong \edgegraph$.
 By the $\Sthree$-symmetry in $u$, the same happens for $\Hy^{AL}$:
 It gets reflected in $v$ and $w'''$ by centralizing elements and not in $u$; 
 therefore 
 $\Hy^{AL}/_{C_\Gamma(\left\langle {AL} \right\rangle)} \cong \edgegraph$.
 
 Summing up, and taking into account that $\Hy/\Gamma$ is contractible,
we obtain the CR orbifold cohomology

$$ \Homol^d_{\rm CR}\left(\left[\Hy^3_\C / \text{PSL}_2\left(\mathbb{Z}\left[\frac{\sqrt{-3}-1}{2}\thinspace\right]\right)\right] \right) \cong 
\begin{cases} 
\rationals, & d = 0, \\
\rationals^{10}, & d=2 ,\\ 
 0, & \mathrm{otherwise}. \end{cases}$$

\subsection{The case $\Gamma = \text{PSL}_2(\ringO_{-11})$.}

\begin{wrapfigure}[15]{r}{54mm}
  \centering
  \includegraphics[width=20mm]{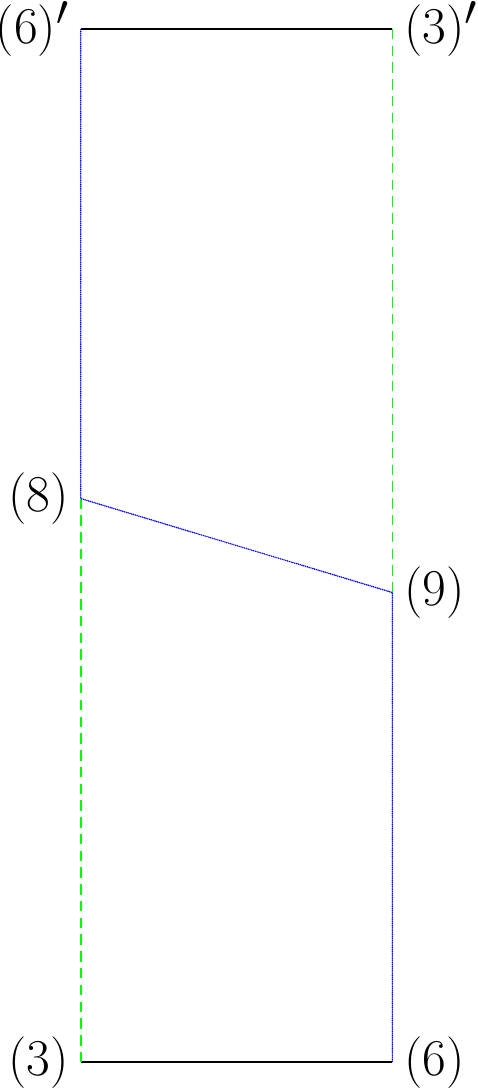}
  \caption{Fundamental domain in the case \mbox{$m = 11$}.} \label{m11}
\end{wrapfigure}

Let $\ringO_{-11}$ be the ring of integers in $\rationals(\sqrt{-11}\thinspace)$.
\\
Then $\ringO_{-11} = \Z[\omega]$ with $\omega = \frac{-1 +\sqrt{-11}}{2}$.

A fundamental domain for $\Gamma := \text{PSL}_2(\ringO_{-11})$ in real hyperbolic 3-space~$\Hy$ has been found by Luigi Bianchi~\cite{Bianchi}.
Half of its lower boundary given in Figure~\ref{m11}. The coordinates of the vertices of Figure~\ref{m11} in the upper-half space model are
$(3) = j$, $(3)' = 1 +\omega +j$, $(6) = \frac{1}{2} +\sqrt{\frac{3}{4}}j$, 
\mbox{$(6)' = \frac{1}{2} +\omega +\sqrt{\frac{3}{4}}j$,}
\mbox{$(8) = \frac{3}{11} +\frac{3}{11}\omega +\sqrt{\frac{2}{11}}j$,} 
\mbox{$(9) = \frac{8}{11} +\frac{5}{11}\omega +\sqrt{\frac{2}{11}}j$.}
The set of representatives of conjugacy classes can be chosen 
$$T = \{ \one,\thinspace \gamma,\thinspace \beta,\thinspace \beta^2 \},$$
with 
$\beta = \pm $ \scriptsize $\begin{pmatrix}0 & -1\\1 & 1\end{pmatrix}$ \normalsize and
$\gamma = $ \scriptsize $\pm \begin{pmatrix}0 & 1\\-1 & 0\end{pmatrix}$, \normalsize \thinspace
\\
so $\gamma$ is of order~2, and $\beta$ is of order~3. 
Using Lemma~\ref{numberOfConjugacyClasses} and with the help of our Bredon homology computations, we check the cardinality of~$T$.
That we have one less conjugacy class of finite order elements than in the case $\ringO_{-2}$, comes from the fact that by Remark~\ref{number of conjugacy classes in finite subgroups}, there is only one conjugacy class of order--2--elements in $\Afour$.

The fixed point sets are then the following subsets of complex hyperbolic space $\Hy := \Hy^3_\C$:
\\
$\mathcal{H}^\one = \mathcal{H}$,
\\
$\mathcal{H}^\gamma = $ the complex geodesic line through $(3)$ and $(8)$,
\\
$\mathcal{H}^\beta = \mathcal{H}^{\beta^2} = $ the complex geodesic line through $(6)$ and $(9)$.

The 2--torsion sub-complex is of homeomorphism type $\edgegraph$ and the 3--torsion sub-complex is of homeomorphism type $\circlegraph$.
Therefore, we obtain

\medskip

$ \Homol^d_{\rm CR}\left([\Hy^3_\C / \text{PSL}_2(\mathbb{Z}[\sqrt{-11}\thinspace])]\right) \cong 
\Homol^d \left(\Hy_\C /_{\text{PSL}_2(\ringO_{-11})}; \thinspace \rationals \right) \oplus
\begin{cases}
\rationals^{1 +2}, & d=2, \\
 \rationals^{2}, & d=3, \\
0, & \mathrm{otherwise}. \end{cases}$

\newpage

\subsection{The case $\Gamma = \text{PSL}_2(\ringO_{-191})$.}

\begin{wrapfigure}[35]{r}{54mm}
  \centering
  \vspace{-10pt}
  \includegraphics[width=38mm]{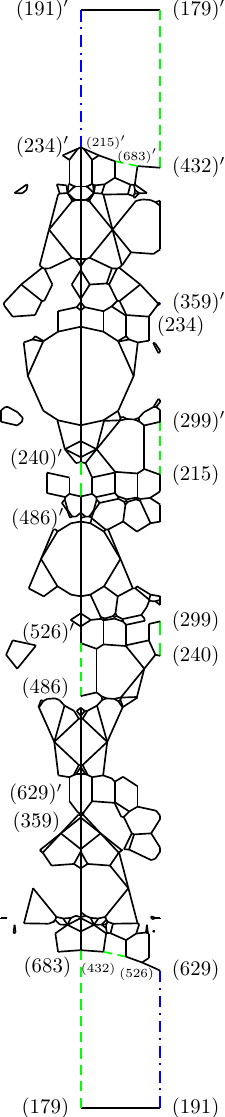}
  \caption{Fundamental domain in the case \mbox{$m = 191$}. The coordinates of the vertices can be displayed by \cite{BianchiGP}.} \label{m191}
  \vspace{-10pt}
\end{wrapfigure}

Let $\ringO_{-191}$ be the ring of integers in $\rationals(\sqrt{-191}\thinspace)$. 
Again, the set of representatives of conjugacy classes can be chosen 
$$T = \{ \one,\thinspace \gamma,\thinspace \beta,\thinspace \beta^2 \},$$
with 
$\beta = \pm $ \scriptsize $\begin{pmatrix}0 & -1\\1 & 1\end{pmatrix}$ \normalsize and
$\gamma = $ \scriptsize $\pm \begin{pmatrix}0 & 1\\-1 & 0\end{pmatrix}$, \normalsize \thinspace
so $\gamma$ is of order~2, and $\beta$ is of order~3.
Both the 2-- and the 3--torsion sub-complexes are of homeomorphism type $\circlegraph$.
Then,

\medskip

$ \Homol^d_{\rm CR}\left([\Hy^3_\C / \text{PSL}_2(\mathbb{Z}[\sqrt{-191}\thinspace])]\right)
\\ \cong 
\Homol^d\left(\Hy_\C /_{\text{PSL}_2(\ringO_{-191})}; \thinspace \rationals \right) \oplus
\begin{cases}
\rationals^{1 +2}, & d=2, \\
 \rationals^{1 +2}, & d=3, \\
0, & \mathrm{otherwise}. \end{cases}$

\bigskip

We conclude this section with the following explanation why in our fundamental domain diagrams,
 there occurs only one representative per torsion-stabilized edge.

\begin{remark}
 Let $e$ be a non-trivially stabilized edge in the fundamental domain for the refined cell complex.
Then the fundamental domain for the $2$--dimensional retract can be chosen such that it contains $e$ as the only edge on its orbit.
\end{remark}
\begin{proof}[Sketch of proof.]
Observe that the inner dihedral angle $\frac{2\pi}{q}$ of the Bianchi fundamental polyhedron is $\frac{2\pi}{\ell}$ or $\frac{\pi}{\ell}$ at its edges admitting a rotation of order $\ell$ from the Bianchi group.
We can verify this in the vertical half-plane where the action of PSL$_2(\Z)$ is embedded into the action of the Bianchi group, for the generators of orders $\ell = 2$ and $\ell = 3$ of PSL$_2(\Z)$ which fix edges orthogonal to the vertical half-plane. 
These angles are transported to all edges stabilized by Bianchi group elements conjugate under SL$_2(\C)$
to these two rotations. Poincar\'e \cite{Poincare} partitions the edges of the Bianchi fundamental polyhedron into cycles, consisting of the edges on the same orbit, of length $\frac{q}{\ell} = 1$ or $2$.
In the case of length $2$, Poincar\'e's description implies that each of the two 2--cells separated by the first edge of the cycle, is respectively on the same orbit as one of the 2--cells separated by the second edge of the cycle. 
As the fundamental domain for the $2$--dimensional retract is strict with respect to the $2$--cells, it can be chosen such that it contains $e$ as the only edge on its orbit.
\end{proof}

Note that we can check our computations using the algorithm of \cite{RahmThesis}*{Section 5.3} 
for the computation of subgroups in the centralizers.

\end{document}